\documentclass[12pt,british]{amsart}
\usepackage[T1]{fontenc}
\usepackage[latin9]{inputenc}
\usepackage[a4paper]{geometry}
\usepackage{babel}
\usepackage{refstyle}
\usepackage{mathrsfs}
\usepackage{amstext}
\usepackage{amsthm}
\usepackage{amssymb}
\usepackage{stmaryrd}
\usepackage[numbers]{natbib}
\usepackage[unicode=true,
 bookmarks=true,bookmarksnumbered=false,bookmarksopen=false,
 breaklinks=true,pdfborder={0 0 1},backref=false,colorlinks=false]
 {hyperref}
\hypersetup{pdftitle={Representation Formulas for Degenerate Inclusions},
 pdfauthor={Y. Capdeboscq, S.C.Y. Ong},
 pdfsubject={Conductivity Equation, Inverse Problems, Polarisation Tensor}}

\makeatletter

\AtBeginDocument{\providecommand\exaref[1]{\ref{exa:#1}}}
\AtBeginDocument{\providecommand\enuref[1]{\ref{enu:#1}}}
\AtBeginDocument{\providecommand\defref[1]{\ref{def:#1}}}
\AtBeginDocument{\providecommand\eqref[1]{\ref{eq:#1}}}
\AtBeginDocument{\providecommand\secref[1]{\ref{sec:#1}}}
\AtBeginDocument{\providecommand\appref[1]{\ref{app:#1}}}
\AtBeginDocument{\providecommand\propref[1]{\ref{prop:#1}}}
\AtBeginDocument{\providecommand\lemref[1]{\ref{lem:#1}}}
\AtBeginDocument{\providecommand\thmref[1]{\ref{thm:#1}}}
\AtBeginDocument{\providecommand\Enuref[1]{\ref{Enu:#1}}}
\AtBeginDocument{\providecommand\Secref[1]{\ref{Sec:#1}}}
\AtBeginDocument{\providecommand\corref[1]{\ref{cor:#1}}}
\RS@ifundefined{subsecref}
  {\newref{subsec}{name = \RSsectxt}}
  {}
\RS@ifundefined{thmref}
  {\def\RSthmtxt{theorem~}\newref{thm}{name = \RSthmtxt}}
  {}
\RS@ifundefined{lemref}
  {\def\RSlemtxt{lemma~}\newref{lem}{name = \RSlemtxt}}
  {}

\numberwithin{equation}{section}
\numberwithin{figure}{section}
\theoremstyle{plain}
\newtheorem{thm}{\protect\theoremname}
\theoremstyle{definition}
\newtheorem{defn}[thm]{\protect\definitionname}
\theoremstyle{plain}
\newtheorem*{assumption*}{\protect\assumptionname}
\theoremstyle{remark}
\newtheorem{rem}[thm]{\protect\remarkname}
\theoremstyle{definition}
\newtheorem{example}[thm]{\protect\examplename}
\theoremstyle{plain}
\newtheorem{lem}[thm]{\protect\lemmaname}
\theoremstyle{plain}
\newtheorem{prop}[thm]{\protect\propositionname}
\theoremstyle{remark}
\newtheorem*{rem*}{\protect\remarkname}
\theoremstyle{remark}
\newtheorem*{notation*}{\protect\notationname}
\theoremstyle{plain}
\newtheorem{cor}[thm]{\protect\corollaryname}
\theoremstyle{plain}
\newtheorem*{cor*}{\protect\corollaryname}
\theoremstyle{remark}
\newtheorem*{acknowledgement*}{\protect\acknowledgementname}

\RS@ifundefined{propref}
  {\def\RSprotxt{proposition~}\newref{prop}{name = \RSprotxt}
   }
  {}
\RS@ifundefined{defref}
  {\def\RSdeftxt{definition~}\newref{def}{name = \RSdeftxt}
 }
  {}
\RS@ifundefined{enuref}
  {\def\RSasutxt{assumption~}\newref{enu}{name = \RSasutxt}
   }
  {}
\RS@ifundefined{corref}
  {\def\RScortxt{corollary~}\newref{cor}{name = \RScortxt}
   }
  {}

\makeatother

\providecommand{\acknowledgementname}{Acknowledgement}
\providecommand{\assumptionname}{Assumption}
\providecommand{\corollaryname}{Corollary}
\providecommand{\definitionname}{Definition}
\providecommand{\examplename}{Example}
\providecommand{\lemmaname}{Lemma}
\providecommand{\notationname}{Notation}
\providecommand{\propositionname}{Proposition}
\providecommand{\remarkname}{Remark}
\providecommand{\theoremname}{Theorem}

\begin{document}
\global\long\def\divx{\mathop{\text{div}}}%

\global\long\def\supp{\mathop{\text{supp}}}%

\global\long\def\dn{d_{n}}%

\global\long\def\qn{A_{n}}%

\global\long\def\rn{B_{n}}%

\title[representation formula for degenerate inclusions]{Extending representation formulae for boundary voltage perturbations
of low volume fraction to very contrasted conductivity inhomogeneities
}
\author{Yves Capdeboscq}
\address{Université de Paris and Sorbonne Université, CNRS, Laboratoire Jacques-Louis
Lions (LJLL), F-75006 Paris, France}
\email{yves.capdeboscq@u-paris.fr}
\author{Shaun Chen Yang Ong}

\maketitle
Imposing either Dirichlet or Neumann boundary conditions on the boundary
of a smooth bounded domain $\Omega$, we study the perturbation incurred
by the voltage potential when the conductivity is modified in a set
of small measure. We consider $\left(\gamma_{n}\right)_{n\in\mathbb{N}}$,
a sequence of perturbed conductivity matrices differing from a smooth
$\gamma_{0}$ background conductivity matrix on a measurable set well
within the domain, and we assume $\left(\gamma_{n}-\gamma_{0}\right)\gamma_{n}^{-1}\left(\gamma_{n}-\gamma_{0}\right)\to0$
in $L^{1}(\Omega)$.  Adapting the limit measure, we show that the
general representation formula introduced for bounded contrasts in
\citep{capdeboscq-vogelius-03a} can be extended to unbounded sequences
of matrix valued conductivities.

\section{The general framework}

Given $d\geq2$, let $\Omega\subset\mathbb{R}^{d}$ be an open, bounded
Lipschitz domain. We study the following family of solutions of perturbed
boundary value problems for the conductivity equation. Given $g\in H^{\frac{1}{2}}\left(\partial\Omega\right)$,
we consider $\left(u_{n}\right)_{n\in\mathbb{N}}\in H^{1}\left(\Omega\right)^{\mathbb{N}}$,
a sequence of perturbations of $u_{0}\in H^{1}\left(\Omega\right)$
given by
\begin{equation}
\begin{cases}
-\divx\left(\gamma_{0}\nabla u_{0}\right) & =0\quad\mbox{in}\quad\Omega,\\
u_{0} & =g\quad\mbox{on\quad\ensuremath{\partial\Omega},}
\end{cases}\text{ and }\begin{cases}
-\divx\big(\gamma_{n}\nabla u_{n}\big) & =0\quad\text{in}\quad\Omega,\\
u_{n} & =g\quad\text{on}\quad\partial\Omega.
\end{cases}\label{eq:diri}
\end{equation}
Alternatively, given $h\in H^{-\frac{1}{2}}\left(\partial\Omega\right)$
with $\int_{\partial\Omega}h\text{d\ensuremath{\sigma=0,}}$ we consider
$\left(u_{n}\right)_{n\in\mathbb{N}}\in H^{1}\left(\Omega\right)^{\mathbb{N}}$,
a sequence of perturbations of $u_{0}\in H^{1}\left(\Omega\right)$
given by
\begin{equation}
\begin{cases}
-\divx\left(\gamma_{0}\nabla u_{0}\right) & =0\quad\mbox{in}\quad\Omega,\\
\gamma_{0}\nabla u_{0}\cdot n & =h\quad\mbox{on\quad\ensuremath{\partial\Omega}},\\
\int_{\partial\Omega}u_{0}\text{d}\sigma & =0,
\end{cases}\text{ and }\begin{cases}
-\divx\big(\gamma_{n}\nabla u_{n}\big) & =0\quad\text{in}\quad\Omega,\\
\gamma_{n}\nabla u_{n}\cdot n & =h\quad\mbox{on\quad\ensuremath{\partial\Omega}},\\
\int_{\partial\Omega}u_{n}\text{d}\sigma & =0.
\end{cases}\label{eq:Neum}
\end{equation}
The conductivity coefficients are assumed to be symmetric positive
definite matrix-valued functions with $\gamma_{0}\in W_{\text{loc}}^{2,d}\left(\mathbb{R}^{d};\mathbb{R}^{d\times d}\right)$,
$\gamma_{n}\in L^{\infty}\left(\Omega;\mathbb{R}^{d\times d}\right)$,
and they satisfy the ellipticity condition 
\[
\lambda_{0}|\zeta|^{2}\leq\gamma_{0}\zeta\cdot\zeta\leq\Lambda_{0}|\zeta|^{2}\quad\text{and}\quad\lambda_{n}|\zeta|^{2}\leq\gamma_{n}\zeta\cdot\zeta\leq\Lambda_{n}|\zeta|^{2},\quad\forall\zeta\in\mathbb{R}^{d},
\]
with $0<\lambda_{n}<\Lambda_{n}$ for all $n\in\mathbb{N}$.

\begin{defn}
\label{def:defdn} Given $\left(\qn\right)_{n\in\mathbb{N}}$ and
$\left(\rn\right)_{n\in\mathbb{N}}$, two sequences of measurable
subsets of $\Omega$ whose Lebesgue measures tend to zero, we define
$\dn\in L^{\infty}\left(\Omega;\mathbb{R}^{d\times d}\right)$, a
positive semi-definite matrix valued function by 
\[
d_{n}=\left(\gamma_{n}+\gamma_{0}\gamma_{n}^{-1}\gamma_{0}\right)1_{A_{n}\cup B_{n}}.
\]
\end{defn}

We make the following assumptions on the inclusion sets.
\begin{assumption*}
We assume that the following assumptions are satisfied:
\end{assumption*}
\begin{enumerate}
\item \label{enu:well-within} There exists $K$ an open subset of $\Omega$
with $C^{\infty}$ boundary such that $d(\partial K,\partial\Omega)>0$
and
\[
\bigcup_{n\in\mathbb{N}}\left(\qn\cup\rn\right)\subset K.
\]
\item \label{enu:pertubative} The perturbation vanishes asymptotically
in $L^{1}\left(\Omega\right)$, that is, 
\[
\left\Vert \dn\right\Vert _{L^{1}\left(\Omega\right)}\leq1\text{ and }\lim_{n\to\infty}\left\Vert \dn\right\Vert _{L^{1}\left(\Omega\right)}=0.
\]
\item \label{enu:notlow} There holds, for all $n\geq1$, 
\[
\gamma_{n}=\gamma_{0}\text{ in }\Omega\setminus\left(\rn\cup\qn\right).
\]
The sets $\qn$ and $\rn$ are disjoint and 
\[
\gamma_{n}\geq\gamma_{0}\text{ a.e. in }A_{n},\quad\gamma_{n}\leq\gamma_{0}\text{ a.e. in }B_{n}
\]
 these inequalities being understood in the sense of quadratic forms.
\item \label{enu:extra} If $\qn\neq\emptyset$ for all $n$, we assume
that one of the following integrability properties are satisfied:
\begin{enumerate}
\item \label{enu:mixed} There exists $p>d$ such that 
\[
\limsup_{n\to\infty}\left\Vert \dn\right\Vert _{L^{p}\left(\qn\right)}<\infty.
\]
\item \label{enu:2d} When $d=2$, for some $p>2$ there holds
\[
\limsup_{n\to\infty}\left\Vert \dn\right\Vert _{L^{p}\left(\rn\right)}<\infty.
\]
\item \label{enu:intertwined} There exists $p>\frac{d}{2}$ such that 
\[
\limsup_{n\to\infty}\left\Vert \dn\right\Vert _{L^{p}\left(\qn\right)}<\infty,
\]
 and there exists $\tau<\frac{1}{d-1}$ such that for all $n\in\mathbb{N},$
\[
d\left(A_{n},B_{n}\right)>\left\Vert \dn\right\Vert _{L^{1}\left(\qn\right)}^{\tau}.
\]
\end{enumerate}
\end{enumerate}
For $f\in L^{p}\left(\Omega\right)$, $1\leq p\leq\infty$, $\left\Vert f\right\Vert _{L^{p}\left(\Omega\right)}$
is the canonical $L^{P}\left(\Omega\right)$ norm. For $U\in L^{p}\left(\Omega;\mathbb{R}^{d}\right)$
the notation $\left\Vert U\right\Vert _{L^{p}\left(\Omega\right)}=\left\Vert \left|U\right|_{d}\right\Vert _{L^{p}\left(\Omega\right)}$
where $\left|\cdot\right|_{d}$ denotes the Euclidean norm in $\mathbb{R}^{d}$.
For $A\in L^{p}\left(\Omega;\mathbb{R}^{d\times d}\right)$, $\left\Vert A\right\Vert _{L^{p}\left(\Omega\right)}$
means $\left\Vert \left|A\right|_{F}\right\Vert _{L^{p}\left(\Omega\right)}$
where $\left|\cdot\right|_{F}$ is the Frobenius norm, that is, the
Euclidean norm on $\mathbb{R}^{d\times d}$. We remind the reader
that $\left|AU\right|_{d}\leq\left|A\right|_{F}\left|U\right|_{d}$
a.e. in $\Omega$, even though the Frobenius norm isn't the subordinate
matrix norm associated with the Euclidean distance in $\mathbb{R}^{d}$.
If $A$ and $B$ are non negative symmetric semi-definite matrices
such that $A\leq B$ in the sense of quadratic forms, then $\left|A\right|_{F}\le\left|B\right|_{F}$.
\begin{rem}
Definition~\ref{def:defdn} implies that on $\qn\cup B_{n},$ 
\[
d_{n}=\left(\gamma_{n}-\gamma_{0}\right)\gamma_{n}^{-1}\left(\gamma_{n}-\gamma_{0}\right)+2\gamma_{0}.
\]
Thus 
\[
d_{n}\geq2\gamma_{0}\text{ and }d_{n}>\left(\gamma_{n}-\gamma_{0}\right)\gamma_{n}^{-1}\left(\gamma_{n}-\gamma_{0}\right).
\]
For all $x\in\rn$, $d_{n}>\gamma_{0}\geq\gamma_{n}\geq\gamma_{0}-\gamma_{n}\geq0$.
For all $x\in\qn$, $d_{n}=\gamma_{n}+\gamma_{0}\gamma_{n}^{-1}\gamma_{0}\geq\gamma_{n}\geq\gamma_{n}-\gamma_{0}$.
All in all, there holds

\begin{equation}
\begin{cases}
\left|\dn\right|_{F} & \geq\left|\gamma_{0}\right|_{F}\\
\left|\dn\right|_{F} & \geq\left|\gamma_{n}\right|_{F}\\
\left|\dn\right|_{F} & \geq\left|\gamma_{n}-\gamma_{0}\right|_{F}\\
\left|\dn\right|_{F} & \geq\text{\ensuremath{\left|\left(\gamma_{n}-\gamma_{0}\right)\gamma_{n}^{-1}\left(\gamma_{n}-\gamma_{0}\right)\right|}}_{F}
\end{cases}\text{a.e. on }\qn\cup\rn.\label{eq:notlowalpha}
\end{equation}
We will use these estimates frequently.
\end{rem}

\begin{rem}
Assumption~\ref{enu:well-within} comes from the fact that near the
boundary of the domain, the behaviour of the solution is different,
as the imposed boundary condition plays an increased role.

Assumption~\ref{enu:pertubative} is sufficient and sharp in general.
Example~\exaref{countrex} illustrates the fact that for some inclusions
$u_{n}\not\to u_{0}$ when $\left\Vert d_{n}\right\Vert {}_{L^{1}\left(\Omega\right)}\not\to0$.

Assumption~\ref{enu:notlow} imposes a limitation for anisotropic
conductivities since $\qn\cap\rn=\emptyset$ : there cannot be an
anisotropic inclusion which is very large in one direction and very
small in another. In the case of isotropic materials, it is simply
means that the inhomogeneities are located in $\qn$ and $\rn$.

Assumption~\ref{enu:extra} imposes additional integrability properties
for $\dn$ only on highly conductive inclusions, not on insulating
ones, in general. If $\qn=\emptyset$, \enuref{extra} is always satisfied.
In dimension two, in the presence of both insulating and conductive
inclusions, if they are arbitrarily mixed, an extra integrability
of either of the two types of inclusions suffices. Alternatively,
if the insulating and conductive inclusions are not too finely intertwined,
a weaker integrability condition is required. While any of the conditions
listed under \enuref{extra} is sufficient for our results to hold,
it is not clear that an assumption is necessary.  As far as the authors
are aware, this is the first result allowing both very insulating
and very highly conductive inclusions.
\end{rem}

For any $y\in\Omega$, the Green function $G(\cdot,y)$ is the weak
solution to the boundary value problem given by
\begin{align*}
\divx\big(\gamma_{0}\nabla G(\cdot,y)\big) & =\delta_{y}\text{ in }\Omega\\
G(\cdot,y) & =0\text{ on }\text{\ensuremath{\partial\Omega}}
\end{align*}
 where $\delta_{y}$ denotes the Dirac measure at the point $y$,
and the Neumann function $N\left(\cdot,y\right)$ is the weak solution
to the boundary value problem given by 
\begin{align*}
\divx\big(\gamma_{0}\nabla N(\cdot,y)\big) & =\delta_{y}\text{ in }\Omega\\
\gamma_{0}\nabla N(\cdot,y)\cdot n & =\frac{1}{\left|\partial\Omega\right|}\text{ on }\text{\ensuremath{\partial\Omega}}.
\end{align*}
The main result of this article is that the general representation
formula introduced in \citep{capdeboscq-vogelius-03a} can be extended
to this context. This result was presented in a preliminary form in
\citep{SCYONGTHESIS}.

\begin{thm}
\label{thm:main}Let $\dn$ be given by \defref{defdn}. Suppose that
assumptions \ref{enu:well-within}, \ref{enu:pertubative}, \ref{enu:notlow}
and \ref{enu:extra} hold. Then, there exists a subsequence also denoted
by $\dn$ and a matrix valued function $M\in L^{2}\big(\Omega,\mathbb{R}^{d\times d};\text{d\ensuremath{\mu}\ensuremath{\big)}}$,
where $\mu$ is the Radon measure generated by the sequence $\frac{1}{\|\dn\|_{L^{1}(\Omega)}}\left|\dn\right|_{F}$,
such that for any $y\in\overline{\Omega\setminus K}$,
\begin{itemize}
\item if $u_{n}$ and $u_{0}$ are solutions to \eqref{diri} there holds
\[
u_{n}\left(y\right)-u_{0}\left(y\right)=\|\dn\|_{L^{1}(\Omega)}\int_{\Omega}M_{ij}\left(x\right)\frac{\partial u_{0}}{\partial x_{i}}\left(x\right)\frac{\partial G\left(x,y\right)}{\partial x_{j}}\text{d}\mu(x)+r_{n}(y),
\]
\item if $u_{n}$ and $u_{0}$ are solutions to \eqref{Neum} there holds
\[
u_{n}\left(y\right)-u_{0}\left(y\right)=\|\dn\|_{L^{1}(\Omega)}\int_{\Omega}M_{ij}\left(x\right)\frac{\partial u_{0}}{\partial x_{i}}\left(x\right)\frac{\partial N\left(x,y\right)}{\partial x_{j}}\text{d}\mu(x)+r_{n}^{\prime}(y),
\]
\end{itemize}
in which $r_{n}\in L^{\infty}\big(\overline{\Omega\setminus K}\big)$
(respectively $r_{n}^{\prime}\in L^{\infty}\big(\overline{\Omega\setminus K}\big)$
) satisfies $\frac{\|r_{n}\|_{L^{\infty}(\Omega\setminus K)}}{\|\dn\|_{L^{1}(\Omega)}}\to0$
(resp. $\frac{\|r_{n}^{\prime}\|_{L^{\infty}(\Omega\setminus K)}}{\|\dn\|_{L^{1}(\Omega)}}\to0$
) uniformly in $g\in H^{\frac{1}{2}}(\partial\Omega)$ (resp. $h\in H^{-\frac{1}{2}}(\partial\Omega)$
) with $\left\Vert g\right\Vert _{H^{\frac{1}{2}}(\partial\Omega)}\leq1$
satisfies (resp. $\left\Vert h\right\Vert _{H^{-\frac{1}{2}}(\partial\Omega)}\leq1$).

The matrix valued function $M\in L^{2}\left(\Omega,d\mu\right)$ is
symmetric. The tensor $M$ can be written as $M=D-W$, where W satisfies
\[
0\leq W\zeta\cdot\zeta\leq\zeta\cdot\zeta\text{\ensuremath{\quad\mu} a.e. in }\Omega,
\]
and if $\gamma_{n}$ and $\gamma_{0}$ are isotropic, 
\[
0\leq W\zeta\cdot\zeta\leq\frac{1}{\sqrt{d}}\zeta\cdot\zeta\text{\ensuremath{\quad\mu} a.e. in }\Omega.
\]
 whereas $D$ is limit in the sense of measures of $\left\Vert \dn\right\Vert _{L^{1}(\Omega)}^{-1}\left(\gamma_{n}-\gamma_{1}\right).$
\end{thm}

Definition~\ref{def:cndn-1} specifies the matrix valued function
$W\in L^{2}\big(\Omega,\mathbb{R}^{d\times d};\text{d\ensuremath{\mu}\ensuremath{\big)}}$.
The tensor $M$ is, up to a factor, the polarisation tensor introduced
in \citep{capdeboscq-vogelius-03a}. Its properties are briefly discussed
in \secref{Properties}, following \citep{capdeboscq-vogelius-06}.

The question of large contrast limits has been considered by other
authors. In \citep{nguyen-vogelius-09}, the authors consider the
case of diametrically bounded inclusions. In \citep{MR3592165}, the
authors consider thin inhomogeneities. Unlike what is done in these
articles, we do not go beyond the perturbation regime. On the other
hand, in this work no geometric assumption is made on the shape of
the inhomogeneities.

To document the sharpness of \enuref{pertubative}, the following
example shows that it may happen that the asymptotic limit of $u_{n}$
is different from $u$ for some sequence $\left(\gamma_{n}\right)_{n\in\mathbb{N}}$
when $\left\Vert \dn\right\Vert _{L^{1}\left(\Omega\right)}\not\to0$
even though $\left|\qn\cup\rn\right|\to0.$
\begin{example}
\label{exa:countrex}Suppose that $\Omega=B(0,2)\subset\mathbb{R}^{d}$,
choose $\qn=B\left(0,1+\frac{1}{n}\right)\setminus B\left(0,1-\frac{1}{n}\right)$,
and $g=x_{1}$. Then for $\gamma_{0}=I_{d}$, the unperturbed solution
of \eqref{diri} corresponds to $u=x_{1}$.

Suppose that $\gamma_{n}$ is radial and constant on $\left(I_{i}\right)_{i\leq1\leq4}$,
where $I_{1}=\left(0,1-\frac{1}{n}\right)$, $I_{2}=\left(1-\frac{1}{n},1\right)$.
$I_{3}=\left(1,1+\frac{1}{n}\right)$, $I_{4}=\left(1+\frac{1}{n},2\right)$,
with values
\[
\gamma_{n}=\chi_{I_{1}\cup I_{4}}+n^{\alpha}\chi_{I_{2}}+n^{\beta}\chi_{I_{3}},
\]
where $\alpha,\beta$ are real parameters. Then, 
\[
\int_{\Omega}\left|\dn\right|_{F}\text{d}x=\sqrt{d}\left(n^{\alpha-1}+n^{-\alpha-1}+n^{\beta-1}+n^{-\beta-1}\right)
\]
 and the solution $u_{n}$ of \eqref{diri} takes the form 
\[
u_{n}=\sum_{i=1}^{4}A_{i}^{n}x_{1}\mathbf{1}_{I_{i}}\left(\left|x\right|\right)+\left|x\right|^{-d}\sum_{i=2}^{4}B_{i}^{n}x_{1}\mathbf{1}_{I_{i}}\left(\left|x\right|\right),
\]
for some constants $\left(A_{i}^{n}\right)_{1\leq i\leq4}$ and $\left(B_{i}^{n}\right)_{2\leq i\leq4}$.
As $n\to\infty$, then $u_{n}\to v$ pointwise where $v=\left(\lim_{n\to\infty}A_{1}^{n}\right)x_{1}$
for $x<1$ and $v=\left(\lim_{n\to\infty}A_{4}^{n}\right)x_{1}+\left(\lim_{n\to\infty}B_{4}^{n}\right)\left|x\right|^{-d}x_{1}$
for $x>\frac{1}{2}$. Computing the value of the constants, we find
that $\left(\lim_{n\to\infty}A_{1}^{n}\right)=\left(\lim_{n\to\infty}A_{4}^{n}\right)=1$
and $\left(\lim_{n\to\infty}B_{4}^{n}\right)=0$ if and only if $-1<\alpha<1$
and $-1<\beta<1$. We further note that if we write $\delta=\min\left(1+\alpha,1+\beta,1-\alpha,1-\beta\right)>0$,
$u_{n}-x_{1}$ is of order $n^{-\delta}$. Written in a slightly different
form, there exists a positive constant $C$ depending on $\alpha$,$\beta$
and $d$ but independent of $n$ such that for all $n\geq1$ there
holds

\[
C^{-1}\int_{\Omega}\left|\dn\right|_{F}\text{d}x\leq\left\Vert u_{n}-x\right\Vert _{L^{1}\left(\Omega\right)}\text{ and }\left\Vert u_{n}-x\right\Vert _{L^{\infty}\left(\Omega\right)}\leq C\int_{\Omega}\left|\dn\right|_{F}\text{d}x.
\]
In this family of examples, the assumption $\int_{\Omega}\left|\dn\right|_{F}\text{d}x\to0$
is necessary for the perturbation regime to exist.
\end{example}

Following the steps in \citep{capdeboscq-vogelius-03a}, the asymptotic
formula that we derive makes use of
\begin{enumerate}
\item A limiting Radon measure $\mu$ which describes the geometry of the
limiting set,
\item A background fundamental solution $G(x,y)$,
\item A limit vector $\mathcal{M}\in\left[L^{2}(\Omega,\text{d}\mu)\right]^{d}$
which describes the variations of the field $\nabla u_{n}$ in the
presence of inhomogeneity sets,
\item A polarisation tensor $M$, independent of $u_{n}$, $u_{0}$, the
larger domain $\Omega$ and the type of boundary condition, such that
$\mathcal{M}=M\nabla u_{0}$ in $L^{2}(\Omega,\text{d}\mu)$.
\end{enumerate}
This will be particularly familiar to readers acquainted to the subsequent
article \citep{capdeboscq-vogelius-06} where an energy-based approach
is also used. It turns out that under \enuref{well-within} and \enuref{pertubative}
only, we can express the first order expansion in terms of $\mathcal{M}$.

Given $u_{n},u_{0}\in H^{1}\left(\Omega\right)$ given by \eqref{diri}
or \eqref{Neum}, we define $w_{n}=u_{n}-u_{0}\in X$ where $X=H_{0}^{1}(\Omega)$
for the Dirichlet problem and $X=\left\{ \phi\in H^{1}\left(\Omega\right):\int_{\Omega}\phi\text{\,d}x=0\right\} $
for the Neumann problem. Here, $w_{n}$ is the weak solution of 
\begin{equation}
\int_{\Omega}\gamma_{n}\nabla w_{n}\cdot\nabla\phi\text{d}x=\int_{\Omega}\left(\gamma_{0}-\gamma_{n}\right)\nabla u_{0}\cdot\nabla\phi\text{\text{d}x for all }\phi\in X.\label{eq:defwn0}
\end{equation}
Note that if $u_{0}$ is the background solution of \eqref{diri}
or \eqref{Neum}, then by classical regularity results \citep[theorem 2.1]{zbMATH00961093},
$u_{0}\in H^{1}\left(\Omega\right)\cap C^{1}(K)$ and $\left\Vert u_{0}\right\Vert _{C^{1}\left(K\right)}\leq C\left(\Omega\right)\left\Vert g\right\Vert _{H^{\frac{1}{2}}\left(\partial\Omega\right)}$,
or $\left\Vert u_{0}\right\Vert _{C^{1}\left(K\right)}\leq C\left(\Omega\right)\left\Vert h\right\Vert _{H^{-\frac{1}{2}}\left(\partial\Omega\right)}$
respectively.
\begin{lem}
\label{lem:defmu} Let $\dn\in L^{\infty}\left(\Omega;\mathbb{R}^{d\times d}\right)$
be given by \defref{defdn}. Then, the sequence $\frac{\left|\dn\right|_{F}}{\left\Vert \dn\right\Vert _{L^{1}(\Omega)}}$
converges up to the possible extraction of a subsequence, in the sense
of measures to a positive radon measure $\mu$, that is,
\begin{equation}
\int_{\Omega}\frac{1}{\left\Vert \dn\right\Vert _{L^{1}(\Omega)}}\left|\dn\right|_{F}\phi\,\text{d}x\to\int_{\Omega}\phi\,d\mu\text{ for all }\phi\in C(\overline{\Omega}).\label{eq:defnmu}
\end{equation}
For each $i,j\in\{1,\ldots,d\}^{2}$, $\frac{1}{\left\Vert \dn\right\Vert _{L^{1}(\Omega)}}\left(\gamma_{n}-\gamma_{0}\right)_{ij}$
converges in the sense of measures to a limit $D_{ij}\in\left[L^{2}(\Omega,\text{d}\mu)\right]$
\begin{equation}
\int_{\Omega}\frac{1}{\left\Vert \dn\right\Vert _{L^{1}(\Omega)}}\left(\gamma_{n}-\gamma_{0}\right)_{ij}\phi\,\text{d}x\to\int_{\Omega}D_{ij}\,\phi\,d\mu\text{ for all }\phi\in C(\overline{\Omega}).\label{eq:measuretensor-1}
\end{equation}
\end{lem}

\begin{proof}
See appendix~\appref{AppendixA}.
\end{proof}
\begin{rem}
The sequence $\left\Vert \dn\right\Vert _{L^{1}(\Omega)}^{-1}\left|\dn\right|_{F}$
only converges to a given measure after extraction of a subsequence
a priori. In the case of an isotropic, constant, conductivity in the
inclusions, $\left\Vert \dn\right\Vert _{L^{1}(\Omega)}^{-1}\left|\dn\right|_{F}=1_{\qn\cup\rn}\left|\qn\cup\rn\right|^{-1}$,
and this measure does not depend on the values taken by $\gamma_{n}$
or $\gamma_{0}$ on $\qn\cup\rn$.

The quantity $\dn$ appears in the following energy estimate.
\end{rem}

\begin{prop}
\label{prop:EnergyEstimateGal} The weak solution of \eqref{defwn0}
$w_{n}\in X$ satisfies 
\begin{equation}
E\left(w_{n}\right):=\int_{\Omega}\gamma_{n}\nabla w_{n}\cdot\nabla w_{n}\text{d}x\leq\left\Vert \dn\right\Vert _{L^{1}(\Omega)}\left\Vert \nabla u_{0}\right\Vert _{L^{\infty}(K)}^{2}.\label{eq:energyest}
\end{equation}
As a consequence, there holds
\begin{equation}
\left\Vert \left(\gamma_{n}-\gamma_{0}\right)\nabla w_{n}\right\Vert _{L^{1}(\Omega)}\leq\left\Vert \dn\right\Vert _{L^{1}(\Omega)}\left\Vert \nabla u_{0}\right\Vert _{L^{\infty}(K)}.\label{eq:L1grad}
\end{equation}
 Furthermore, up to the possible extraction of a subsequence, $\frac{1}{\left\Vert \dn\right\Vert _{L^{1}(\Omega)}}\left(\gamma_{0}-\gamma_{n}\right)\nabla w_{n}$
converges in the sense of measures to a limit
\begin{equation}
\int_{\Omega}\frac{1}{\left\Vert \dn\right\Vert _{L^{1}(\Omega)}}\left(\gamma_{0}-\gamma_{n}\right)\nabla w_{n}\cdot\Psi\,\text{d}x\to\int_{\Omega}\mathcal{W}\cdot\Psi\,d\mu,\label{eq:measuretensor}
\end{equation}
where $\mathcal{W}\in\left[L^{2}(\Omega,\text{d}\mu)\right]^{d}$
and $\mu$ is given by \eqref{defnmu}.
\end{prop}

\begin{rem}
The upper estimates \eqref{energyest} and \eqref{L1grad} are sharp
with respect to the order of dependence on $\left\Vert \dn\right\Vert _{L^{1}(\Omega)}$
as shown in example~\exaref{countrex}.
\end{rem}

\begin{proof}
The proof of \propref{EnergyEstimateGal} is similar to the moderate
contrast case in \citep{capdeboscq-vogelius-03a}, but with estimates
in terms of $\left\Vert \dn\right\Vert _{L^{1}(\Omega)}$. It is provided
in appendix \appref{proofape}.
\end{proof}
Under \enuref{notlow} an improved Aubin--Céa--Nitsche estimate
can be derived (\lemref{L2bound}), which allows to consider extreme
contrast and depends on the $L^{1}$ norm of $\dn$ only. This allows
in particular to show independence with respect to the domain and
the prescribed boundary condition, as stated below (see also \citep[lemma 1]{capdeboscq-vogelius-06}).
\begin{lem}
\label{lem:boundaryindep} Suppose that assumptions \ref{enu:well-within},
\ref{enu:pertubative}, and \ref{enu:notlow} hold. Let $\tilde{\Omega}$
be any bounded regular open set such that $K\subset\tilde{\Omega}$
with $\text{dist}(K,\tilde{\Omega})>0$. Let $Y$ be one of the spaces
\[
H_{0}^{1}(\tilde{\Omega}),\quad\tilde{H}^{1}(\tilde{\Omega}):=\left\{ \phi\in H^{1}\big(\tilde{\Omega}\big):\int_{\tilde{\Omega}\setminus K}\phi\text{\,d}x=0\right\} 
\]
 or 
\[
H_{\#}^{1}(\tilde{\Omega}):=\left\{ \phi\in H_{\text{loc}}^{1}\left(\mathbb{R}^{d}\right)\,:\,\int_{\tilde{\Omega}\setminus K}\phi\text{\,d}x=0\text{ and }\phi\quad\tilde{\Omega}-\text{periodic}\right\} ,
\]
the latter if $\tilde{\Omega}$ is a cube. We write the weak solution
of (\ref{eq:defwn0}) $w_{n}^{X}\in X$ and we set $w_{n}^{Y}$ to
be the unique weak solution to 
\begin{equation}
\int_{Q}\gamma_{n}\nabla w_{n}^{Y}\cdot\nabla\phi\text{d}x=\int_{Q}\left(\gamma_{0}-\gamma_{n}\right)\nabla u_{0}\cdot\nabla\phi\text{d}x\text{ for all }\phi\in Y,\label{eq:varfom}
\end{equation}
then for any $\tau\in\left(0,\frac{1}{2\left(d-1\right)}\right),$
there exists $C>0$ which may depend on $\tau$, $\text{\ensuremath{\Omega}}$,
$K$, $\Lambda_{0}$, $\lambda_{0}$ and $\left\Vert \gamma_{0}\right\Vert _{W^{2,d}\left(\Omega\right)}$
only such that 
\[
\frac{1}{\|\dn\|_{L^{1}(\Omega)}}\|\big(\gamma_{n}-\gamma_{0}\big)\nabla\left(w_{n}^{Y}-w_{n}^{X}\right)\|_{L^{1}(\Omega)}\leq C\left\Vert \dn\right\Vert _{L^{1}(\Omega)}^{\text{\ensuremath{\tau}}}\left\Vert \nabla u_{0}\right\Vert _{L^{\infty}\left(\Omega\right)}.
\]
As a consequence, the measured valued vector $\mathcal{M}^{X}$ and
$\mathcal{M}^{Y}$ obtained from any two of these variational problems
via \propref{EnergyEstimateGal} are equal.
\end{lem}

The proof of this result is provided in \secref{APE-gen}. It now
suffices to focus on Dirichlet problem to establish \thmref{main}.
To prove polarisability, that is, $\mathcal{M}=M\nabla u_{0}$, our
argument requires one of the additional requirements detailed in \enuref{extra}.
\begin{defn}
\label{def:cndn-1} For each $i=1,\ldots,d$, we define the correctors
$w_{n}^{i}\in H_{0}^{1}(\Omega)$ as the weak solutions of 
\begin{equation}
\int_{\Omega}\gamma_{n}\nabla w_{n}^{i}\cdot\nabla\phi\text{\,d}x=\int_{\Omega}\left(\gamma_{0}-\gamma_{n}\right)\mathbf{e}_{i}\cdot\nabla\phi\text{\,d}x\text{ for all }\phi\in H_{0}^{1}\left(\Omega\right).\label{eq:defwn0-1-1}
\end{equation}
We call $W_{ij}\in L^{2}\left(\Omega,d\mu\right)$ the scalar weak$^{*}$
limit of $\frac{1}{\|\dn\|_{L^{1}(\Omega)}}\left(\nabla w_{n}^{i}\cdot\left(\gamma_{0}-\gamma_{n}\right)\mathbf{e}_{j}\right).$
\end{defn}

\begin{rem*}
The connection between this tensor and its parent introduced in \citep{capdeboscq-vogelius-03a}
is discussed in \secref{Properties}.
\end{rem*}
\begin{prop}
\label{prop:Polar-Outline} Suppose assumptions \Enuref{well-within},
\Enuref{pertubative}, \Enuref{notlow} and \Enuref{extra} are satisfied.
Given $\Omega^{\prime}$ a smooth open subset of $\Omega$ containing
$K$ such that $d\left(\Omega^{\prime},\partial\Omega\right)>\frac{1}{3}d\left(K,\partial\Omega\right)$
and $d\left(K,\partial\Omega^{\prime}\right)>\frac{1}{3}d\left(K,\partial\Omega\right)$,
there holds 
\[
\int_{\Omega}\left(\gamma_{n}-\gamma_{0}\right)\nabla w_{n}\cdot\nabla x_{i\,}\phi\,\text{d}x=\int_{\Omega}\left(\gamma_{n}-\gamma_{0}\right)\nabla w_{n}^{i}\cdot\nabla u_{0\,}\phi\,\text{d}x+\int_{\Omega}r_{n}\cdot\nabla\phi\text{\,d}x
\]
with 
\[
\left\Vert r_{n}\right\Vert _{L^{1}\left(\Omega\right)}\leq C\left\Vert \dn\right\Vert _{L^{1}\left(\Omega\right)}^{1+\eta}\left(\left\Vert \nabla u_{0}\right\Vert _{L^{\infty}(K)}+\left\Vert u_{0}\right\Vert _{L^{\infty}\left(\partial\Omega^{\prime}\right)}\right),
\]
where the positive constants $C$ and $\eta$ may depend only on $\tau$,
$\text{\ensuremath{\Omega}}$, $K$, $\left\Vert \gamma_{0}\right\Vert _{W^{2,d}\left(\Omega\right)}$,
$\Lambda_{0}$, $\lambda_{0}$, and possibly $\left\Vert \dn\right\Vert _{L^{p}\left(\qn\right)}$
or $\left\Vert \dn\right\Vert _{L^{p}\left(\rn\right)}$ for some
$p$ depending on which of the alternatives listed in \enuref{extra}
is satisfied.
\end{prop}

\begin{proof}
The proof of \propref{Polar-Outline} is the purpose of \secref{ProofPolar}.
Depending on whether both insulating and conducting inhomogeneities
are present, and whether the dimension is $2$ or more, it is the
combined conclusion of \propref{insulatingpolar}, \propref{polarstream}
and \propref{notwinning}.
\end{proof}
We are now in position to conclude the proof of \thmref{main}, but
for the properties of the polarisation tensor $M$, left for \lemref{-ReussBounds}.
\begin{proof}[End of the proof of \thmref{main}]
 Consider the Dirichlet case. Observing that the weak formulation
for the solution $w_{n}=u_{n}-u_{0}$ reads 
\begin{equation}
\int_{\Omega}\gamma_{0}\nabla w_{n}\cdot\nabla\phi\text{d}x=\int_{\Omega}\left(\gamma_{0}-\gamma_{n}\right)\left(\nabla w_{n}+\nabla u_{0}\right)\cdot\nabla\phi\,\text{d}x\label{eq:11}
\end{equation}
for any $\phi\in H_{0}^{1}(\Omega),$ we choose a sequence $\phi_{m}\in C_{c}^{1}(\Omega)$
such that $\phi_{m}\to G_{y}$ in $W^{1,1}\big(\Omega\big)$ and $\phi_{m}\to\nabla G_{y}$
in $C^{0}\big(K)$. Using the fact that $w_{n}$ is smooth away from
the set $K$ and the fact that $\gamma_{n}-\gamma_{0}$ is supported
in $K$, we may insert $\phi_{m}$ into (\ref{eq:11}) and pass to
the limit to conclude that
\[
\int_{\Omega}\gamma_{0}\nabla w_{n}\cdot\nabla_{x}G(x,y)\text{\,d}x=\int_{\Omega}\left(\gamma_{0}-\gamma_{n}\right)\left(\nabla u_{0}+\nabla w_{n}\right)\cdot\nabla_{x}G\left(x,y\right)\text{\,d}x.
\]
After an integration by parts we obtain 
\begin{eqnarray*}
\left(u_{n}-u_{0}\right)\left(y\right) & = & \int_{\Omega}\left(\gamma_{n}-\gamma_{0}\right)\left(\nabla w_{n}+\nabla u_{0}\right)\cdot\nabla_{x}G(x,y)\text{\,d}x\\
 & = & \|\dn\|_{L^{1}(\Omega)}\int_{\Omega}\frac{1}{\|\dn\|_{L^{1}(\Omega)}}\left(\gamma_{n}-\gamma_{0}\right)\nabla u_{0}\cdot\nabla_{x}G\left(x,y\right)\text{\,d}x\\
 & - & \|\dn\|_{L^{1}(\Omega)}\int_{\Omega}\frac{1}{\|\dn\|_{L^{1}(\Omega)}}\left(\gamma_{0}-\gamma_{n}\right)\nabla w_{n}\cdot\nabla_{x}G\left(x,y\right)\text{\,d}x
\end{eqnarray*}
Using the fact that $\forall y\in\overline{\Omega\setminus K}\quad\text{and}\quad\forall x\in\cup_{n=1}^{\infty}\left(\qn\cup\rn\right)$,
we may find a smooth function $\phi_{y}\in C^{0}\big(\overline{\Omega}\big)$
such that
\[
\phi_{y}(x)=\nabla_{x}G\left(x,y\right)\quad\text{\ensuremath{\forall x\in K},}
\]
and thanks to \propref{Polar-Outline}, and \lemref{defmu}, we have
\[
\left(u_{n}-u_{0}\right)\left(y\right)=\|\dn\|_{L^{1}(\Omega)}\int_{\Omega}\left(D_{ij}-W_{ij}\right)\frac{\partial u_{0}}{\partial x_{i}}\frac{\partial G\left(x,y\right)}{\partial x_{j}}\text{d}\mu(x)+r_{n}(y),
\]
where $W\in L^{2}\big(\Omega,\mathbb{R}^{d\times d};\text{d}\mu\big)$
is introduced in \defref{cndn-1}. Note that $\phi_{y}$ is uniformly
bounded $\forall\left(x,y\right)\in K\times\overline{\Omega\setminus K}$
. Moreover, the remainder estimate from \propref{Polar-Outline} only
depends on $\|g\|_{H^{\frac{1}{2}}\left(\partial\Omega\right)},$
therefore $\|r_{n}\|_{L^{\infty}(\Omega)}/\|\dn\|_{L^{1}(\Omega)}$
converges to $0$ uniformly in $y\in\overline{\Omega\setminus K}$
and $g$ in the unit ball of the space $H^{\frac{1}{2}}\left(\partial\Omega\right)$.
The Neumann case is similar.
\end{proof}
The rest of paper is structured as follows. In \secref{APE-gen} we
derive a number of a priori estimates, and prove \lemref{boundaryindep}.
\Secref{ProofPolar} is devoted to the proof of \propref{Polar-Outline}.
In \secref{Properties} we briefly discuss some of the properties
of the tensor $M$, and prove \lemref{-ReussBounds}. Finally in \secref{An-example}
we show with an example that the a priori bounds for $M$ given in
\thmref{main} are attained.

\section{\label{sec:APE-gen}Proof of \lemref{boundaryindep} and a priori
estimates.}
\begin{lem}
\label{lem:Linftysimple} Given $\Omega^{\prime}$ a smooth domain
as defined in \propref{Polar-Outline}, there holds 
\begin{align*}
\left\Vert u_{n}\right\Vert _{L^{\infty}\left(\partial\Omega^{\prime}\right)}+\left\Vert \nabla u_{n}\right\Vert _{L^{\infty}\left(\partial\Omega^{\prime}\right)} & \leq C\left(\left\Vert \nabla u_{0}\right\Vert _{L^{\infty}(K)}+\left\Vert u_{0}\right\Vert _{L^{\infty}\left(\partial\Omega^{\prime}\right)}\right),\\
\left\Vert w_{n}\right\Vert _{L^{\infty}\left(\partial\Omega^{\prime}\right)}+\left\Vert \nabla w_{n}\right\Vert _{L^{\infty}\left(\partial\Omega^{\prime}\right)} & \leq C\left\Vert w_{n}\right\Vert _{L^{2}\left(\Omega\setminus K\right)}
\end{align*}
where $C>0$ depends on $\Omega^{\prime},$ $K,\Omega$, $\Lambda_{0}$,
$\lambda_{0}$ and $\left\Vert \gamma_{0}\right\Vert _{W^{2,d}\left(\Omega\right)}$
only. Furthermore, 
\begin{equation}
\left\Vert w_{n}\right\Vert _{L^{\infty}\left(K\right)}\leq C\left(\left\Vert \nabla u_{0}\right\Vert _{L^{\infty}(K)}+\left\Vert u_{0}\right\Vert _{L^{\infty}\left(\partial\Omega^{\prime}\right)}\right).\label{eq:linftybasic}
\end{equation}
\end{lem}

\begin{notation*}
In the sequel, we use the notation $a\lesssim b$ to mean $a\leq Cb$,
where $C$ is a constant, possibly changing from line to line depending
on the parameters announced in the claim we wish to prove.
\end{notation*}
\begin{proof}
Let $\Omega^{\prime}$ and $\Omega^{\prime\prime}$ be two open domains
such that $K\subset\Omega^{\prime\prime}\subset\Omega^{\prime}\subset\Omega$,
with $9d(\Omega^{\prime\prime},\partial\Omega^{\prime})>$$d(K,\partial\Omega)$
and 9$d(K,\partial\Omega^{\prime\prime})>d(K,\partial\Omega)$. Since
\[
-\divx\left(\gamma_{0}\nabla w_{n}\right)=0\quad\text{on}\quad\Omega^{\prime\prime}\setminus\Omega^{\prime}
\]
and $\gamma_{0}\in W^{2,d}\left(\Omega\right)$, classical regularity
theory shows that 
\begin{equation}
\left\Vert w_{n}\right\Vert _{C^{1}\left(\overline{\Omega^{\prime}\setminus\Omega^{\prime\prime}}\right)}\lesssim\left\Vert w_{n}\right\Vert _{L^{2}\left(\Omega\setminus K\right)}.\label{eq:classwn}
\end{equation}
By Poincaré's inequality (or Poincaré-Wirtinger's inequality depending
on $X$) since $w_{n}$ vanishes on $\partial\Omega$, there holds
\[
\left\Vert w_{n}\right\Vert _{L^{2}\left(\Omega\setminus K\right)}\lesssim\left\Vert \nabla w_{n}\right\Vert _{L^{2}\left(\Omega\setminus K\right)}.
\]
 On the other hand, using the fact that $\gamma_{n}=\gamma_{0}\geq\lambda_{0}I_{d}$
on $\Omega\setminus K,$ there holds 
\begin{align*}
\left\Vert \nabla w_{n}\right\Vert _{L^{2}\left(\Omega\setminus K\right)} & \leq\frac{1}{\sqrt{\lambda_{0}}}\left(E\left(w_{n}\right)\right)^{\frac{1}{2}}\\
 & \lesssim\left\Vert \dn\right\Vert _{L^{1}(\Omega)}^{\frac{1}{2}}\left\Vert \nabla u_{0}\right\Vert _{L^{\infty}(K)}\\
 & \lesssim\left\Vert \nabla u_{0}\right\Vert _{L^{\infty}(K)},
\end{align*}
where we used \eqref{energyest} for the penultimate inequality and
the fact that the sequence $\left\Vert \dn\right\Vert _{L^{1}(\Omega)}$
is bounded on the last line. Therefore on $\Omega\setminus\Omega^{\prime}$,
the function $w_{n}$ satisfies $\divx\left(\gamma_{0}\nabla w_{n}\right)=0$
with $\left|w_{n}\right|\lesssim\left\Vert \nabla u_{0}\right\Vert _{L^{\infty}(K)}$
on $\partial\Omega^{\prime}$ and satisfies a homogeneous boundary
condition on $\partial\Omega$ (or periodicity). By comparison, this
implies 
\[
\left\Vert w_{n}\right\Vert _{L^{\infty}\left(\Omega\setminus\Omega^{\prime}\right)}\lesssim\left\Vert \nabla u_{0}\right\Vert _{L^{\infty}(K)}.
\]
Furthermore, $u_{n}=w_{n}+u_{0}$ satisfies $\left\Vert u_{n}\right\Vert _{C^{1}\left(\partial\Omega^{\prime}\right)}\leq\left\Vert w_{n}\right\Vert _{C^{1}\left(\partial\Omega^{\prime}\right)}+\left\Vert u_{0}\right\Vert _{C^{1}\left(\partial\Omega^{\prime}\right)}$.
Finally, since $\divx\left(\gamma_{n}\nabla u_{n}\right)=0$ on $\Omega^{\prime}$,
by comparison $\left\Vert u_{n}\right\Vert _{L^{\infty}\left(\Omega^{\prime}\right)}=\left\Vert u_{n}\right\Vert _{C\left(\partial\Omega^{\prime}\right)}$,
and $\left\Vert w_{n}\right\Vert _{L^{\infty}\left(\Omega\right)}\leq\left\Vert w_{n}\right\Vert _{L^{\infty}\left(\Omega^{\prime}\right)}+\left\Vert w_{n}\right\Vert _{L^{\infty}\left(\Omega\setminus\Omega^{\prime}\right)}\lesssim\left\Vert \nabla u_{0}\right\Vert _{L^{\infty}(K)}+\left\Vert u_{0}\right\Vert _{L^{\infty}\left(\partial\Omega^{\prime}\right)}$
and the conclusion follows.
\end{proof}
Following the strategy introduced in \citep{capdeboscq-vogelius-03a},
we now show that the potential tends to zero faster than the gradient
via an Aubin--Céa--Nitsche argument. The novelty of this result
is that it depends on $\gamma_{n}$ only via on $\left\Vert \dn\right\Vert _{L^{1}(\Omega)}$.
\begin{lem}
\label{lem:L2bound} For any $\tau\in\left[1,\frac{d}{d-1}\right)$,
and given $\Omega^{\prime}$ a smooth domain as defined in \propref{Polar-Outline},
there holds 
\begin{equation}
\left\Vert w_{n}\right\Vert _{L^{2}(\Omega)}\leq C\left\Vert \dn\right\Vert _{L^{1}(\Omega)}^{\frac{\tau}{2}}\left(\left\Vert \nabla u_{0}\right\Vert _{L^{\infty}(K)}+\left\Vert u_{0}\right\Vert _{L^{\infty}\left(\partial\Omega^{\prime}\right)}\right),\label{eq:improveL2general}
\end{equation}
with the constant $C$ may depend on $\tau$, $\text{\ensuremath{\Omega}, \ensuremath{K}}$,
$\left\Vert \gamma_{0}\right\Vert _{W^{2,d}\left(\Omega\right)}$
, and the a priori bounds $\Lambda_{0}$ and $\lambda_{0}$ only.
\end{lem}

\begin{proof}
Consider the following auxiliary equation 
\begin{eqnarray}
-\divx\left(\gamma_{0}\nabla\psi_{n}\right) & = & w_{n}\quad\text{in}\quad\text{\ensuremath{\Omega}}\label{eq:adjoint}\\
\psi_{n} & = & 0\quad\text{on}\quad\partial\Omega.\nonumber 
\end{eqnarray}
Since $\gamma_{0}\in W^{2,d}\left(\Omega;\mathbb{R}^{d\times d}\right)$
we infer from elliptic regularity theory (see e.g. \citep{zbMATH00961093})
that for any $q\geq2$, the solution $\psi_{n}$ satisfies
\begin{equation}
\left\Vert \psi_{n}\right\Vert _{W^{2,q}(\Omega)}+\left\Vert \psi_{n}\right\Vert _{W^{1,q}(\Omega)}\lesssim\left\Vert w_{n}\right\Vert _{L^{q}(\Omega)}.\label{eq:boundH2q}
\end{equation}
Testing \eqref{adjoint} with $w_{n}$, and recalling that $\supp\left(\gamma_{n}-\gamma_{0}\right)\subset\left(\qn\cup\rn\right)\subset K$,
an integration by parts shows 
\begin{align}
\left\Vert w_{n}\right\Vert _{L^{2}(\Omega)}^{2} & =\int_{\Omega}\gamma_{0}\nabla\psi_{n}\cdot\nabla w_{n}\,\text{d}x\nonumber \\
 & =\int_{\Omega}\left(\gamma_{0}-\gamma_{n}\right)\nabla w_{n}\cdot\nabla\psi_{n}\,\text{d}x+\int_{\Omega}\gamma_{n}\nabla\psi_{n}\cdot\nabla w_{n}\,\text{d}x\nonumber \\
 & =\int_{\qn\cup\rn}\left(\gamma_{0}-\gamma_{n}\right)\nabla w_{n}\cdot\nabla\psi_{n}+\int_{\qn\cup\rn}\left(\gamma_{0}-\gamma_{n}\right)\nabla u_{0}\cdot\nabla\psi_{n}\label{eq:AubinNitsche}
\end{align}
Using Cauchy--Schwarz, we find
\begin{align*}
 & \int_{\qn\cup\rn}\left(\gamma_{0}-\gamma_{n}\right)\nabla w_{n}\cdot\nabla\psi_{n}\,\text{d}x\\
\leq & \bigg(\int_{\qn\cup\rn}\gamma_{n}\nabla w_{n}\cdot\nabla w_{n}\,\text{d}x\bigg)^{\frac{1}{2}}\bigg(\int_{\Omega}\dn\nabla\psi_{n}\cdot\nabla\psi_{n}\,\text{d}x\bigg)^{\frac{1}{2}},
\end{align*}
and thanks to \eqref{energyest}, 
\[
\int_{\qn\cup\rn}\left(\gamma_{0}-\gamma_{n}\right)\nabla w_{n}\cdot\nabla\psi_{n}\,\text{d}x\leq\left\Vert \dn\right\Vert _{L^{1}\left(\Omega\right)}\left\Vert \nabla u_{0}\right\Vert _{L^{\infty}\left(K\right)}\left\Vert \nabla\psi_{n}\right\Vert _{L^{\infty}\left(K\right)}.
\]
Similarly, using \eqref{notlowalpha},
\begin{align*}
 & \int_{\qn\cup\rn}\left(\gamma_{0}-\gamma_{n}\right)\nabla u_{0}\cdot\nabla\psi_{n}\,\text{d}x\\
\leq & \bigg(\int_{\qn\cup\rn}\gamma_{n}\nabla u_{0}\cdot\nabla u_{0}\,\text{d}x\bigg)^{\frac{1}{2}}\bigg(\int_{\Omega}\dn\nabla\psi_{n}\cdot\nabla\psi_{n}\,\text{d}x\bigg)^{\frac{1}{2}}\\
\leq & \left\Vert \dn\right\Vert _{L^{1}\left(\Omega\right)}\left\Vert \nabla u_{0}\right\Vert _{L^{\infty}\left(K\right)}\left\Vert \nabla\psi_{n}\right\Vert _{L^{\infty}\left(K\right)}.
\end{align*}
 and \eqref{AubinNitsche} becomes 
\begin{equation}
\left\Vert w_{n}\right\Vert _{L^{2}(\Omega)}^{2}\leq2\left\Vert \dn\right\Vert _{L^{1}\left(\Omega\right)}\left\Vert \nabla u_{0}\right\Vert _{L^{\infty}\left(K\right)}\left\Vert \nabla\psi_{n}\right\Vert _{L^{\infty}\left(K\right)}.\label{eq:AubinN-2}
\end{equation}
On the other hand, choosing $q=d+\epsilon$ in \eqref{boundH2q} there
holds 
\begin{equation}
\left\Vert \nabla\psi_{n}\right\Vert _{L^{\infty}(\Omega)}\lesssim\left\Vert \psi_{n}\right\Vert _{W^{2,d+\epsilon}\left(\Omega\right)}\lesssim\left\Vert w_{n}\right\Vert _{L^{d+\epsilon}(\Omega)}\label{regular H2}
\end{equation}
By interpolation, and using the a priori bound (\ref{eq:linftybasic})
for $w_{n}$ given in \lemref{Linftysimple}, we find
\begin{equation}
\left\Vert w_{n}\right\Vert _{L^{d+\epsilon}(\Omega)}\leq\left\Vert w_{n}\right\Vert _{L^{2}(\Omega)}^{\frac{2}{d+\epsilon}}\left\Vert w_{n}\right\Vert _{L^{\infty}(\Omega)}^{1-\frac{2}{d+\epsilon}}\lesssim\left\Vert w_{n}\right\Vert _{L^{2}(\Omega)}^{\frac{2}{d+\epsilon}}\left(\left\Vert \nabla u_{0}\right\Vert _{L^{\infty}(K)}+\left\Vert u_{0}\right\Vert _{L^{\infty}\left(\partial\Omega^{\prime}\right)}\right)^{1-\frac{2}{d+\epsilon}}.\label{eq:AubinN-4}
\end{equation}
Combining \eqref{AubinN-2}, \ref{regular H2}, and \eqref{AubinN-4},
we have obtained 
\begin{align*}
\left\Vert w_{n}\right\Vert _{L^{2}(\Omega)}^{2\left(1-\frac{1}{d+\epsilon}\right)} & \lesssim\left\Vert \dn\right\Vert _{L^{1}\left(\Omega\right)}\left\Vert \nabla u_{0}\right\Vert _{L^{\infty}\left(K\right)}\left(\left\Vert \nabla u_{0}\right\Vert _{L^{\infty}(K)}+\left\Vert u_{0}\right\Vert _{L^{\infty}\left(\partial\Omega^{\prime}\right)}\right)^{1-\frac{2}{d+\epsilon}}\\
 & \lesssim\left\Vert \dn\right\Vert _{L^{1}\left(\Omega\right)}\left(\left\Vert \nabla u_{0}\right\Vert _{L^{\infty}(K)}+\left\Vert u_{0}\right\Vert _{L^{\infty}\left(\partial\Omega^{\prime}\right)}\right)^{2\left(1-\frac{1}{d+\epsilon}\right)},
\end{align*}
which is equivalent to \eqref{improveL2general}.
\end{proof}
\begin{rem}
Note that estimate \eqref{improveL2general} improves on previous
estimates, even in the case of bounded contrasts (see \citep[lemma 1]{capdeboscq-vogelius-03a}).
It is arbitrarily close to the estimate one obtains for a fixed, scaled
shape with constant scalar conductivity \citep{ammari-kang-04}.
\end{rem}

\begin{cor}
For any $q\geq2$ and any $\tau\in\left[1,\frac{d}{d-1}\right)$ ,
with the same notations as in \lemref{L2bound}, there holds
\begin{equation}
\left\Vert w_{n}\right\Vert _{L^{q}(\Omega)}\leq C\left\Vert \dn\right\Vert _{L^{1}(\Omega)}^{\frac{\tau}{q}}\left(\left\Vert \nabla u_{0}\right\Vert _{L^{\infty}(K)}+\left\Vert u_{0}\right\Vert _{L^{\infty}\left(\partial\Omega^{\prime}\right)}\right).\label{eq:Lsbound}
\end{equation}
Furthermore, $w_{n}$ solution of \eqref{defwn0} satisfies
\begin{equation}
\left\llbracket \nabla w_{n}\right\rrbracket _{L^{\infty}\left(\partial\Omega^{\prime}\right)}+\left\llbracket w_{n}\right\rrbracket _{L^{\infty}\left(\partial\Omega^{\prime}\right)}\leq C\left\Vert \dn\right\Vert _{L^{1}\left(\Omega\right)}^{\frac{\tau}{2}}\left(\left\Vert \nabla u_{0}\right\Vert _{L^{\infty}(K)}+\left\Vert u_{0}\right\Vert _{L^{\infty}\left(\partial\Omega^{\prime}\right)}\right).\label{eq:bdylinfykprime}
\end{equation}
\end{cor}

\begin{proof}
We write 
\[
\left\Vert w_{n}\right\Vert _{L^{s}(\Omega)}\leq\left\Vert w_{n}\right\Vert _{L^{2}(\Omega)}^{\frac{2}{q}}\left\Vert w_{n}\right\Vert _{L^{\infty}(\Omega)}^{1-\frac{2}{q}}
\]
and estimate \eqref{Lsbound} follows from \eqref{improveL2general}
and \eqref{linftybasic}. Estimate \eqref{bdylinfykprime} follows
from \lemref{Linftysimple} and \lemref{L2bound}.
\end{proof}
We now address the independence of the polarisation tensor $M$ from
the boundary conditions.
\begin{proof}[Proof of \lemref{boundaryindep}]
. Given $\tau=\left(0,\frac{1}{2}\frac{1}{d-1}\right),$ Following
the steps of \ref{lem:L2bound} with \eqref{varfom} and $w_{n}^{Y}$,
we find
\begin{equation}
\left\Vert w_{n}^{Y}\right\Vert _{L^{2}\left(\tilde{\Omega}\right)}\lesssim\left\Vert \dn\right\Vert _{L^{1}(\Omega)}^{\frac{1+2\tau}{2}}\left(\left\Vert \nabla u_{0}\right\Vert _{L^{\infty}(K)}+\left\Vert u_{0}\right\Vert _{L^{\infty}\left(\partial\Omega^{\prime}\right)}\right).\label{eq:impvl2}
\end{equation}
Now, we choose a smooth cut-off function $\chi\in C_{c}^{\infty}\left(\tilde{\Omega}\right)$
such that $\chi=1$ on $K$. Noting that $\divx\left(\gamma_{n}\nabla\left(w_{n}^{X}-w_{n}^{Y}\right)\right)=0$
on $\hat{\Omega},$ Caccioppoli's inequality writes 
\[
\int_{\tilde{\Omega}}\gamma_{n}\nabla\left(\chi\left(w_{n}^{Y}-w_{n}^{X}\right)\right)\cdot\nabla\left(\chi\left(w_{n}^{Y}-w_{n}^{X}\right)\right)\text{d}x=\int_{\tilde{\Omega}\setminus K}\left(\gamma_{0}\nabla\chi\cdot\nabla\chi\right)\left(w_{n}^{Y}-w_{n}^{X}\right)^{2}\text{d}x,
\]
that is,
\begin{align*}
\int_{\tilde{\Omega}}\gamma_{n}\nabla\left(w_{n}^{Y}-w_{n}^{X}\right)\cdot\nabla\left(w_{n}^{Y}-w_{n}^{X}\right)\text{d}x & \leq C\left(\tilde{\Omega},K\right)\left(\left\Vert w_{n}^{X}\right\Vert _{L^{2}\left(\Omega\right)}^{2}+\left\Vert w_{n}^{Y}\right\Vert _{L^{2}\left(\tilde{\Omega}\right)}^{2}\right),\\
 & \lesssim\left\Vert \dn\right\Vert _{L^{1}(\Omega)}^{1+2\tau}\left(\left\Vert \nabla u_{0}\right\Vert _{L^{\infty}(K)}+\left\Vert u_{0}\right\Vert _{L^{\infty}\left(\partial\Omega^{\prime}\right)}\right)^{2}
\end{align*}
This in turn shows, by Cauchy-Schwarz,
\begin{align*}
\|\big(\gamma_{n}-\gamma_{0}\big)\nabla\left(w_{n}^{Y}-w_{n}^{X}\right)\|_{L^{1}(\Omega)} & \leq\left\Vert \dn\right\Vert _{L^{1}(\Omega)}^{\frac{1}{2}}\Bigg(\int_{K}\gamma_{n}\nabla\left(w_{n}^{Y}-w_{n}^{X}\right)\cdot\nabla\left(w_{n}^{Y}-w_{n}^{X}\right)\text{d}x\Bigg)^{\frac{1}{2}}\\
 & \leq C\left(\tilde{\Omega},K\right)\|\dn\|_{L^{1}\left(\Omega\right)}^{1+\tau}\left(\left\Vert \nabla u_{0}\right\Vert _{L^{\infty}(K)}+\left\Vert u_{0}\right\Vert _{L^{\infty}\left(\partial\Omega^{\prime}\right)}\right)
\end{align*}
As a result, $\frac{1}{\left\Vert \dn\right\Vert _{L^{1}\left(\Omega\right)}}\left\Vert \left(\gamma_{n}-\gamma_{0}\right)\nabla\left(w_{n}^{Y}-w_{n}^{X}\right)\right\Vert _{L^{1}(\Omega)}\to0$,
which implies equivalence that the limiting measures resulting from
$\frac{1}{\left\Vert \dn\right\Vert _{L^{1}\left(\Omega\right)}}\left(\gamma_{n}-\gamma_{0}\right)\nabla w_{n}^{X}$
and $\frac{1}{\left\Vert \dn\right\Vert _{L^{1}\left(\Omega\right)}}\left(\gamma_{n}-\gamma_{0}\right)\nabla w_{n}^{Y}$
are equal.
\end{proof}

\section{\label{sec:ProofPolar}Proof of \propref{Polar-Outline}}

We use the following corollary to the a priori energy estimate given
in \propref{EnergyEstimateGal}.
\begin{cor*}[Corollary to \propref{EnergyEstimateGal}]
 For any $p\geq1$, there holds 
\begin{equation}
\left\Vert \gamma_{n}\nabla w_{n}\right\Vert _{L^{\frac{2p}{p+1}}\left(\qn\right)}\leq d^{\frac{1}{4}}\left\Vert \dn\right\Vert _{L^{1}\left(\Omega\right)}^{\frac{1}{2}}\left\Vert \dn\right\Vert _{L^{p}\left(\qn\right)}^{\frac{1}{2}}\left\Vert \nabla u_{0}\right\Vert _{L^{\infty}(K)}.\label{eq:lpboundflux}
\end{equation}
\end{cor*}
\begin{proof}
Using Hölder's inequality, it holds that for any $p\geq1$
\begin{equation}
\left\Vert \gamma_{n}\nabla w_{n}\right\Vert _{L^{\frac{2p}{p+1}}\left(\qn\right)}\leq\left\Vert \gamma_{n}^{\frac{1}{2}}\right\Vert _{L^{2p}\left(\qn\right)}\left(E\left(w_{n}\right)\right)^{\frac{1}{2}}.\label{eq:l2bdgrad0-1}
\end{equation}
We have
\[
\left\Vert \gamma_{n}^{\frac{1}{2}}\right\Vert _{L^{2p}\left(\qn\right)}=\left(\int_{\qn}\left|\gamma_{n}^{\frac{1}{2}}\right|_{F}^{2p}\text{d}x\right)^{\frac{1}{2p}},
\]
and, using the fact that for $d\times d$ symmetric matrix $A$, $\left|A^{2}\right|_{F}\leq\left|A\right|_{F}^{2}\leq\sqrt{d}\left|A^{2}\right|_{F},$we
find, using \eqref{notlowalpha},
\begin{equation}
\left\Vert \gamma_{n}^{\frac{1}{2}}\right\Vert _{L^{2p}\left(\qn\right)}\leq d^{\frac{1}{4}}\left(\int_{\qn}\left|\gamma_{n}\right|_{F}^{p}\text{d}x\right)^{\frac{1}{2p}}=d^{\frac{1}{4}}\left\Vert \gamma_{n}\right\Vert _{L^{p}\left(\qn\right)}^{\frac{1}{2}}\leq d^{\frac{1}{4}}\left\Vert \dn\right\Vert _{L^{p}\left(\qn\right)}^{\frac{1}{2}}.\label{eq:l2bdgrad1-1}
\end{equation}
Putting together \eqref{energyest}, \eqref{l2bdgrad0-1} and \eqref{l2bdgrad1-1}
the conclusion follows.
\end{proof}
The following error estimate is a key tool for the proof of ~\propref{Polar-Outline}
\begin{prop}
\label{prop:remainderestimate} For any $\phi\in C^{1}(\overline{\Omega})$,
there holds 
\begin{align}
 & \int_{\Omega}\left(\Big(\gamma_{n}-\gamma_{0}\Big)\nabla w_{n}\cdot\nabla x_{i}\right)\phi\,\text{d}x\label{eq:remainder}\\
= & \int_{\Omega}\left(\Big(\gamma_{n}-\gamma_{0}\Big)\nabla w_{n}^{i}\cdot\nabla u_{0}\right)\phi\,\text{d}x+\int_{\Omega}r_{n}\cdot\nabla\phi\,\text{d}x\nonumber 
\end{align}
with $r_{n}\in L^{1}(\Omega)$. Furthermore for any $\tau\in\left[1,\frac{2d-1}{2d-2}\right)$,
the following estimate holds
\begin{equation}
\left|\int_{\Omega}r_{n}\cdot\nabla\phi\,\text{d}x\right|\leq C\left\Vert \nabla\phi\right\Vert _{L^{\infty}(\Omega)}\left(\left\Vert \dn\right\Vert _{L^{1}\left(\Omega\right)}^{\tau}\left(\left\Vert \nabla u_{0}\right\Vert _{L^{\infty}(K)}+\left\Vert u_{0}\right\Vert _{L^{\infty}\left(\partial\Omega^{\prime}\right)}\right)+\varepsilon_{n}\right),\label{remainder2}
\end{equation}
The constant $C$ may depends on $\tau$, $\text{\ensuremath{\Omega}, \ensuremath{\Omega^{\prime}}, \ensuremath{K}}$,
$\left\Vert \gamma_{0}\right\Vert _{W^{2,d}\left(\Omega\right)}$,
and the a priori bounds $\Lambda_{0}$ and $\lambda_{0}$ only. The
remainder term $\varepsilon_{n}$ satisfies the following two a priori
estimates
\begin{equation}
\varepsilon_{n}\leq\left\Vert \dn\right\Vert _{L^{1}\left(\Omega\right)}\left(\left\Vert w_{n}\right\Vert _{L^{\infty}\left(\qn\right)}+\left\Vert w_{n}^{i}\right\Vert _{L^{\infty}\left(\qn\right)}\left\Vert \nabla u_{0}\right\Vert _{L^{\infty}(K)}\right)\label{eq:epsilon1linf}
\end{equation}
 and, for $p>d,$
\begin{equation}
\varepsilon_{n}\leq\left\Vert \dn\right\Vert _{L^{1}\left(\Omega\right)}^{1+\eta}\left\Vert \dn\right\Vert _{L^{p}\left(\qn\right)}^{\frac{1}{2}}\left(\left\Vert \nabla u_{0}\right\Vert _{L^{\infty}(K)}+\left\Vert u_{0}\right\Vert _{L^{\infty}\left(\partial\Omega^{\prime}\right)}\right).\label{eq:epsiloonlp}
\end{equation}
 where $\eta>0$ depends only on $p.$
\end{prop}

\begin{rem*}
Note that estimates \eqref{epsilon1linf} and \eqref{epsiloonlp}
imply that $\epsilon_{n}\leq0$ when $\qn=\emptyset.$
\end{rem*}
\begin{proof}
We write $Z$ as a shorthand for $\left\Vert \nabla u_{0}\right\Vert _{L^{\infty}(K)}+\left\Vert u_{0}\right\Vert _{L^{\infty}\left(\partial\Omega^{\prime}\right)}$.
A computation shows that 
\[
\int_{\Omega}\left(\left(\gamma_{n}-\gamma_{0}\right)\nabla w_{n}\cdot\nabla x_{i}\right)\phi\,\text{d}x=\int_{\Omega}\left(\left(\gamma_{0}-\gamma_{n}\right)\nabla w_{n}^{i}\cdot\nabla u_{0}\right)\,\phi\,\text{d}x+\int_{\Omega}r_{n}\cdot\nabla\phi\,\text{d}x
\]
where the remainder term $r_{n}\in L^{1}(\Omega)$ is
\begin{align*}
r_{n} & =\left(\gamma_{n}-\gamma_{0}\right)\left(w_{n}^{i}\nabla u_{0}-w_{n}\nabla x_{i}\right)+w_{n}^{i}\gamma_{n}\nabla w_{n}-w_{n}\gamma_{n}\nabla w_{n}^{i}.
\end{align*}
Now, write $T_{1}=\mathbf{1}_{\gamma_{n}\leq\gamma_{0}}\left(r_{n}\cdot\nabla\phi\right)$
and $T_{2}=r_{n}\cdot\nabla\phi-T_{1}$.

\begin{align*}
\left\Vert T_{1}\right\Vert _{L^{1}\left(\Omega\right)} & \leq\int_{\Omega\cap\left\{ \gamma_{n}\leq\gamma_{0}\right\} }\left|w_{n}\left(\gamma_{n}\nabla w_{n}^{i}\right)\cdot\nabla\phi\right|\,\text{d}x+\int_{\Omega\cap\left\{ \gamma_{n}\leq\gamma_{0}\right\} }\left|w_{n}^{i}\left(\gamma_{n}\nabla w_{n}\right)\cdot\nabla\phi\right|\,\text{d}x\\
 & +\int_{\Omega\cap\left\{ \gamma_{n}\leq\gamma_{0}\right\} }\left|w_{n}^{i}\left(\gamma_{n}-\gamma_{0}\right)\nabla u_{0}\cdot\nabla\phi\right|\,\text{d}x+\int_{\Omega\cap\left\{ \gamma_{n}\leq\gamma_{0}\right\} }\left|w_{n}\left(\gamma_{n}-\gamma_{0}\right)\nabla x_{i}\cdot\nabla\phi\right|\,\text{d}x\\
 & \leq\left\Vert \nabla\phi\right\Vert _{L^{\infty}(\Omega)}\left\Vert \gamma_{0}\right\Vert _{L^{\infty}\left(\Omega\right)}^{\frac{1}{2}}\left(\left\Vert w_{n}\right\Vert _{L^{2}\left(\Omega\right)}E\left(w_{n}^{i}\right)^{\frac{1}{2}}+\left\Vert w_{n}^{i}\right\Vert _{L^{2}\left(\Omega\right)}E\left(w_{n}\right)^{\frac{1}{2}}\right.\\
 & +\left.\left\Vert w_{n}^{i}\right\Vert _{L^{2}\left(\Omega\right)}\left\Vert \dn\right\Vert _{L^{1}\left(\Omega\right)}^{\frac{1}{2}}\left\Vert \nabla u_{0}\right\Vert _{L^{\infty}\left(K\right)}+\left\Vert w_{n}\right\Vert _{L^{2}\left(\Omega\right)}\left\Vert \dn\right\Vert _{L^{1}\left(\Omega\right)}^{\frac{1}{2}}\right)
\end{align*}
Thanks to estimate \eqref{energyest} and \eqref{improveL2general}
(applied to $u_{0}=x_{i}$ for the corrector terms $w_{n}^{i}$) we
find 
\[
\left\Vert T_{1}\right\Vert _{L^{1}\left(\Omega\right)}\lesssim\left\Vert \dn\right\Vert _{L^{1}\left(\Omega\right)}^{\frac{1}{2}+\frac{1}{2}\tau^{\prime}}\left\Vert \nabla\phi\right\Vert _{L^{\infty}(\Omega)}Z,
\]
with $\tau^{\prime}\in\left[1,\frac{d}{d-1}\right)$, so that $\tau=\frac{1+\tau^{\prime}}{2}\in\left[1,\frac{2d-1}{2d-2}\right).$
We now turn to the other term. The triangle inequality gives
\begin{align}
\left\Vert T_{2}\right\Vert _{L^{1}\left(\Omega\right)} & \leq\int_{\qn}\left|w_{n}\left(\gamma_{n}\nabla w_{n}^{i}\right)\cdot\nabla\phi\right|\,\text{d}x+\int_{\qn}\left|w_{n}^{i}\left(\gamma_{n}\nabla w_{n}\right)\cdot\nabla\phi\right|\,\text{d}x\label{eq:bound1}\\
 & +\int_{\qn}\left|w_{n}^{i}\left(\gamma_{n}-\gamma_{0}\right)\nabla u_{0}\cdot\nabla\phi\right|\,\text{d}x+\int_{\qn}\left|w_{n}\left(\gamma_{n}-\gamma_{0}\right)\nabla x_{i}\cdot\nabla\phi\right|\,\text{d}x.\nonumber 
\end{align}
Recall that thanks to \eqref{notlowalpha}, $\left|\gamma_{n}-\gamma_{0}\right|_{F}<\left|\dn\right|_{F}$.
Thus using \eqref{lpboundflux} with $p=1$, and \eqref{notlowalpha},
we deduce from \eqref{bound1} that 
\[
\left\Vert T_{2}\right\Vert _{L^{1}\left(\Omega\right)}\lesssim\left\Vert \dn\right\Vert _{L^{1}\left(\Omega\right)}\left(\left\Vert w_{n}\right\Vert _{L^{\infty}\left(\qn\right)}+\left\Vert w_{n}^{i}\right\Vert _{L^{\infty}\left(\qn\right)}\left\Vert \nabla u_{0}\right\Vert _{L^{\infty}(K)}\right)\left\Vert \nabla\phi\right\Vert _{L^{\infty}(K)},
\]
which corresponds to estimate \eqref{epsilon1linf}.

Alternatively, applying Hölder's inequality, then the $L^{p}$ bound
\eqref{lpboundflux} and the $L^{q}$ bound \eqref{Lsbound} with
the conjugate exponent, we find for any $p\geq1$, and any $\theta\in\left[1,\frac{d}{d-1}\right),$
\begin{align*}
 & \int_{\qn}\left|w_{n}\gamma_{n}\nabla w_{n}^{i}\cdot\nabla\phi\right|\,\text{d}x\\
\leq & \left\Vert \gamma_{n}\nabla w_{n}^{i}\right\Vert _{L^{\frac{2p}{p+1}}(\qn)}\left\Vert w_{n}\right\Vert _{L^{\frac{2p}{p-1}}(\qn)}\left\Vert \nabla\phi\right\Vert _{L^{\infty}(K)}\\
\lesssim & \left\Vert \dn\right\Vert _{L^{1}\left(\Omega\right)}^{\frac{1}{2}}\left\Vert \dn\right\Vert _{L^{p}\left(\qn\right)}^{\frac{1}{2}}\left\Vert \dn\right\Vert _{L^{1}(\Omega)}^{\left(\frac{1}{2}-\frac{1}{2p}\right)\theta}Z\left\Vert \nabla\phi\right\Vert _{L^{\infty}(K)}.
\end{align*}
Similarly
\[
\int_{\qn}\left|w_{n}^{i}\gamma_{n}\nabla w_{n}\cdot\nabla\phi\right|\,\text{d}x\lesssim\left\Vert \dn\right\Vert _{L^{1}\left(\Omega\right)}^{\frac{1}{2}}\left\Vert \dn\right\Vert _{L^{p}\left(\qn\right)}^{\frac{1}{2}}\left\Vert \dn\right\Vert _{L^{1}(\Omega)}^{\left(\frac{1}{2}-\frac{1}{2p}\right)\theta}Z\left\Vert \nabla\phi\right\Vert _{L^{\infty}(K)}.
\]
 Using \eqref{notlowalpha}, Hölder's inequality and the $L^{s}$
bounds \eqref{Lsbound}, we write 
\begin{align*}
\int_{\qn}\left|w_{n}^{i}\left(\gamma_{n}-\gamma_{0}\right)\nabla u_{0}\cdot\nabla\phi\right| & \,\text{d}x\leq\left\Vert \dn^{\frac{1}{2}}\right\Vert _{L^{2}\left(\qn\right)}\left\Vert \dn^{\frac{1}{2}}w_{n}^{i}\right\Vert _{L^{2}\left(\qn\right)}\left\Vert \nabla u_{0}\right\Vert _{L^{\infty}(K)}\left\Vert \nabla\phi\right\Vert _{L^{\infty}(K)}\\
 & \lesssim\left\Vert \dn\right\Vert _{L^{1}\left(\qn\right)}^{\frac{1}{2}}\left\Vert \dn\right\Vert _{L^{p}\left(\qn\right)}^{\frac{1}{2}}\left\Vert w_{n}^{i}\right\Vert _{L^{\frac{2p}{p-1}}\left(\qn\right)}\left\Vert \nabla u_{0}\right\Vert _{L^{\infty}(K)}\left\Vert \nabla\phi\right\Vert _{L^{\infty}(K)}\\
 & \lesssim\left\Vert \dn\right\Vert _{L^{1}\left(\Omega\right)}^{\frac{1}{2}}\left\Vert \dn\right\Vert _{L^{p}\left(\qn\right)}^{\frac{1}{2}}\left\Vert \dn\right\Vert _{L^{1}(\Omega)}^{\left(\frac{1}{2}-\frac{1}{2p}\right)\theta}Z\left\Vert \nabla\phi\right\Vert _{L^{\infty}(K)},
\end{align*}
and by the same argument, 
\[
\int_{\qn}\left|w_{n}\left(\gamma_{n}-\gamma_{0}\right)\nabla x_{i}\cdot\nabla\phi\right|\,\text{d}x\lesssim\left\Vert \dn\right\Vert _{L^{1}\left(\Omega\right)}^{\frac{1}{2}}\left\Vert \dn\right\Vert _{L^{p}\left(\qn\right)}^{\frac{1}{2}}\left\Vert \dn\right\Vert _{L^{1}(\Omega)}^{\left(\frac{1}{2}-\frac{1}{2p}\right)\theta}Z\left\Vert \nabla\phi\right\Vert _{L^{\infty}(K)}.
\]
Altogether, for any $p\geq1$, and any $\theta\in\left[1,\frac{d}{d-1}\right),$
\[
\left\Vert T_{2}\right\Vert _{L^{1}\left(\Omega\right)}\lesssim\left\Vert \dn\right\Vert _{L^{1}\left(\Omega\right)}^{\frac{1}{2}}\left\Vert \dn\right\Vert _{L^{p}\left(\qn\right)}^{\frac{1}{2}}\left\Vert \dn\right\Vert _{L^{1}(\Omega)}^{\left(\frac{1}{2}-\frac{1}{2p}\right)\theta}Z\left\Vert \nabla\phi\right\Vert _{L^{\infty}(K)}.
\]
 For any $p>d$, pick $\theta=\frac{1}{2}\left(\frac{p}{p-1}+\frac{d}{d-1}\right)$,
then 
\[
\eta=\frac{1}{2}\left(\frac{d}{d-1}\frac{p-1}{p}-1\right)>0,
\]
and 
\[
\left\Vert T_{2}\right\Vert _{L^{1}\left(\Omega\right)}\leq\left\Vert \dn\right\Vert _{L^{1}\left(\Omega\right)}^{1+\eta}\left\Vert \dn\right\Vert _{L^{p}\left(\qn\right)}^{\frac{1}{2}}Z\left\Vert \nabla\phi\right\Vert _{L^{\infty}(K)},
\]
which concludes the proof of estimate \eqref{epsiloonlp}.
\end{proof}
\begin{prop}
\label{prop:insulatingpolar} Suppose assumptions \ref{enu:well-within},
\ref{enu:pertubative}, and \ref{enu:notlow} are satisfied. Additionally
assume that either $\qn=\emptyset$, or \enuref{mixed} holds. Given
$\Omega^{\prime}$ a smooth domain as defined in \propref{Polar-Outline},
there holds 
\[
\int_{\Omega}\left(\Big(\gamma_{n}-\gamma_{0}\Big)\nabla w_{n}\cdot\nabla x_{i}\right)\phi\,\text{d}x=\int_{\Omega}\left(\Big(\gamma_{n}-\gamma_{0}\Big)\nabla w_{n}^{i}\cdot\nabla u_{0}\right)\phi\,\text{d}x+\int_{\Omega}r_{n}\cdot\nabla\phi\,\text{d}x
\]
with 
\[
\left\Vert r_{n}\right\Vert _{L^{1}\left(\Omega\right)}\leq C\left\Vert \dn\right\Vert _{L^{1}\left(\Omega\right)}^{1+\eta}\left(\left\Vert \nabla u_{0}\right\Vert _{L^{\infty}(K)}+\left\Vert u_{0}\right\Vert _{L^{\infty}\left(\partial\Omega^{\prime}\right)}\right),
\]
where the positive constants $C$ and $\eta$ may depend only on $\tau$,
$\Omega$, $\ensuremath{K}$, $\left\Vert \gamma_{0}\right\Vert _{W^{2,d}\left(\Omega\right)}$,
$\Lambda_{0}$ and $\lambda_{0}$ and $\left\Vert \dn\right\Vert _{L^{p}\left(\qn\right)}$.
\end{prop}

\begin{proof}
This is an immediate consequence of \propref{remainderestimate}.
\end{proof}

\subsection{The high conductivity inclusion case when $d=2$}

This section addresses the case when \enuref{2d} holds. When $d=2$,
as it is well known, there is a direct relation between high and low
conductivity problem, by means of stream functions (see e.g. \citep{milton-02}).
We use this indirect method to obtain the polarisability result under
\enuref{2d}. We remind the reader of the following classical result.
\begin{lem}[{\citep[Lemma I.1][]{BBH94}}]
 \label{lem:BBH} Let $\Omega$ be any smooth open set in $\mathbb{R}^{2}$,
not necessarily simply connected, and $D$ be a vector field such
that 
\[
\divx\left(D\right)=0\text{ on }\Omega,\text{and }\int_{\text{\ensuremath{\Gamma_{i}}}}D\cdot n\text{d}\sigma=0
\]
 on each connected component $\Gamma_{i}$ of $\partial\Omega$. Then,
there exists a function $H$ such that 
\[
D=\left(-\partial_{x_{2}}H,\partial_{x_{1}}H\right)\text{on }\Omega.
\]
\end{lem}

Let $\text{\ensuremath{\left(\Gamma_{i}\right)_{1\leq i\leq N}}}$
the connected components of $\partial\Omega$ and let $Fb_{n}$ and
$Fb_{0}$ the unique solutions of 
\begin{equation}
\begin{cases}
\divx\left(\gamma_{n}\nabla Fb_{n}\right)=0 & \text{ on }\Omega^{\prime},\\
\gamma_{n}\nabla Fb_{n}\cdot n=\frac{1}{\left|\Gamma_{i}\right|}\int_{\Gamma_{i}}\gamma_{n}\nabla u_{n}\cdot n\text{d}\sigma & \text{ on each \ensuremath{\Gamma_{i}.}}\\
\int_{\Omega}Fb_{n}\text{d}x=0.
\end{cases}\label{eq:Fbn}
\end{equation}
and 
\begin{equation}
\begin{cases}
\divx\left(\gamma_{0}\nabla Fb_{0}\right)=0 & \text{ on }\Omega^{\prime},\\
\gamma_{0}\nabla Fb_{0}\cdot n=\frac{1}{\left|\Gamma_{i}\right|}\int_{\Gamma_{i}}\gamma_{0}\nabla u_{0}\cdot n\text{d}\sigma & \text{ on each \ensuremath{\Gamma_{i}.}}\\
\int_{\Omega}Fb_{0}\text{d}x=0.
\end{cases}\label{eq:Fb0}
\end{equation}

Then applying \lemref{BBH} to $\gamma_{n}\nabla\left(u_{n}-Fb_{n}\right)$
and $\gamma_{0}\nabla\left(u_{0}-Fb_{0}\right)$ there exist stream
functions $\psi_{n},\psi_{0}\in H^{1}\left(\Omega^{\prime}\right)$
such that
\begin{equation}
\gamma_{n}\nabla\left(u_{n}-Fb_{n}\right)=J\nabla\psi_{n}\text{ and }\gamma_{0}\nabla\left(u_{0}-Fb_{0}\right)=J\nabla\psi_{0}\text{ a.e. in }\Omega^{\prime}.\label{eq:defpsi0n}
\end{equation}
where $J$ is the antisymmetric matrix $\left(\begin{array}{cc}
0 & -1\\
1 & 0
\end{array}\right)$. As the stream functions may be chosen uniquely up to an additive
constant, we may assume without loss of generality that they satisfy
the constraint 
\[
\int_{\Omega}\psi_{n}\text{d}x=0=\int_{\Omega}\psi_{0}\text{d}x.
\]
Thus, $\psi_{n}$ and $\psi_{0}$ are weak solutions of 
\begin{align*}
-\divx\left(\sigma_{n}\nabla\psi_{n}\right) & =0\text{ }\text{ in }\Omega^{\prime}\\
-\divx\left(\sigma_{0}\nabla\psi_{0}\right) & =0\text{ }\text{ in }\Omega^{\prime}
\end{align*}
where the conductivity matrices $\sigma_{n}$ and $\sigma_{0}$ are
defined as 
\[
\sigma_{n}:=J^{T}\gamma_{n}^{-1}J\quad\text{and}\quad\sigma_{0}:=J^{T}\gamma_{0}^{-1}J.
\]

When then define $\Sigma_{n}$ as $d_{n}$ was with respect to $\gamma_{0}$
and $\gamma_{n},$that is
\begin{defn}
\label{def:defsigma}We set 
\[
\Sigma_{n}=\left(\sigma_{n}+\sigma_{0}\sigma_{n}^{-1}\sigma_{0}\right)1_{A_{n}\cup B_{n}}.
\]
\end{defn}

\begin{prop}
Given $\Omega^{\prime}$ a smooth domain as defined in \propref{Polar-Outline},
given $\psi_{n}$ and $\psi_{0}$ be the stream functions defined
in \eqref{defpsi0n}. The function $\varphi_{n}=\psi_{n}-\psi_{0}$
satisfies
\begin{equation}
-\divx\left(\sigma_{n}\nabla\varphi_{n}\right)=\divx\left(\left(\sigma_{n}-\sigma_{0}\right)\nabla\psi_{0}\right)\text{ in }\mathcal{D}^{\prime}\left(\text{\ensuremath{\Omega^{\prime}}}\right)\label{eq:varphinv}
\end{equation}
and for any $\tau\in\left(0,\frac{1}{2}\right)$ there holds
\begin{equation}
\left\Vert \sigma_{n}\nabla\varphi_{n}\cdot\nu\right\Vert _{H^{-\frac{1}{2}}\left(\partial\Omega^{\prime}\right)}\leq C\left\Vert \dn\right\Vert _{L^{1}(\Omega)}^{\frac{1}{2}+\tau}\left\Vert g\right\Vert _{H^{\frac{1}{2}}\left(\partial\Omega^{\prime}\right)},\label{eq:improveL2general-fluxinv}
\end{equation}
where the constant $C$ may depend only on $\tau$, $\Omega$, $\ensuremath{K}$,
$\left\Vert \gamma_{0}\right\Vert _{W^{2,d}\left(\Omega\right)}$
, $\Lambda_{0}$ and $\lambda_{0}$ .
\end{prop}

\begin{proof}
Thanks to \eqref{defpsi0n}, since $d\left(\partial\Omega^{\prime},K\right)>0$,
on $\partial\Omega^{\prime}$ 
\begin{align*}
\sigma_{n}\nabla\varphi_{n} & =\sigma_{0}\nabla\varphi_{n}=J^{T}\nabla\left(u_{n}-Fb_{n}-u_{0}-Fb_{0}\right)\\
 & =J^{T}\nabla\left(w_{n}+Fb_{n}-Fb_{0}\right).
\end{align*}
Thanks to estimate \eqref{bdylinfykprime} applied to $w_{n}$ and
to $Fb_{n}$ and $Fb_{0}$, there holds 
\[
\left\Vert \sigma_{n}\nabla\varphi_{n}\right\Vert _{L^{\infty}\left(\partial\Omega^{\prime}\right)}\leq C\left\Vert \dn\right\Vert _{L^{1}\left(\Omega\right)}^{\frac{1}{2}+\tau}\left(\left\Vert \nabla u_{0}\right\Vert _{L^{\infty}(K)}+\left\Vert u_{0}\right\Vert _{L^{\infty}\left(\partial\Omega^{\prime}\right)}\right).
\]
which implies \eqref{improveL2general-fluxinv}.
\end{proof}
We note that the role of $\rn$ and $\qn$ are swapped when considering
\eqref{varphinv} rather than \eqref{defwn0}. The polarisability
for $\varphi_{n}$ is therefore established from \propref{insulatingpolar}
provided $\left\Vert \dn\right\Vert _{L^{p}\left(\rn\right)}<\infty$
for some $p>2.$
\begin{cor}
\label{cor:lowcontrastbyhighcontrast-1}Suppose that Assumptions \ref{enu:well-within},
\ref{enu:pertubative}, and \ref{enu:notlow} are satisfied. Additionally
assume that $d=2$ and for some $p>2,$
\[
\limsup_{n}\left\Vert \Sigma_{n}\right\Vert _{L^{p}\left(\rn\right)}<\infty.
\]
The function $\varphi_{n}=\psi_{n}-\psi_{0}$, the weak solution to
\eqref{varphinv}, satisfies
\begin{equation}
\frac{1}{\left\Vert \Sigma_{n}\right\Vert _{L^{1}(\Omega)}}\Big(\sigma_{0}-\sigma_{n}\Big)\nabla\varphi_{n}\,\text{dx}\overset{*}{\quad{\rightharpoonup}\quad}\tilde{N}\nabla\psi_{0}\,\text{d\ensuremath{\nu}}\label{lowcontrastbyhighcontrast}
\end{equation}
in the space of bounded Radon measures where $\tilde{N}\in L^{2}\big(\Omega,\mathbb{R}^{d\times d};\text{d\ensuremath{\nu}\ensuremath{\big)}},$
and $\nu$ is the Radon measure generated by the sequence $\frac{1}{\|\Sigma_{n}\|_{L^{1}(\Omega)}}\Sigma_{n}$.
The convergence is uniform with respect to $g\in H^{1/2}\left(\partial\Omega\right)$
provided $\left\Vert g\right\Vert _{H^{\frac{1}{2}}\left(\partial\Omega\right)}\leq1$.
\end{cor}

\begin{proof}
The proof follows directly from \propref{remainderestimate} and \lemref{boundaryindep}.
\end{proof}
\begin{lem}
The symmetric positive definite matrix $\Sigma_{n}$ given by \defref{defsigma}
satisfies 
\[
\Sigma_{n}=\sigma_{n}+\sigma_{0}\sigma_{n}^{-1}\sigma_{0}=J^{T}\gamma_{0}^{-1}\dn\gamma_{0}^{-1}J.
\]
As a consequence, denoting $\nu$ and $\mu$ to be the Radon measures
generated by the sequences $\frac{\Sigma_{n}}{\left\Vert \Sigma_{n}\right\Vert _{L^{1}(\Omega)}}$
and $\frac{\dn}{\left\Vert \dn\right\Vert _{L^{1}(\Omega)}}$ respectively,
the Radon-Nikodym derivatives $\frac{\text{d}\nu}{\text{d\ensuremath{\mu}}}$
and $\frac{\text{d}\mu}{\text{d\ensuremath{\nu}}}$ belongs to $L^{\infty}\left(\Omega;\text{d}\mu\right)$
and $L^{\infty}\left(\Omega;\text{d}\nu\right)$ respectively, and
the spaces $L^{p}\left(\Omega;\text{d}\mu\right)$ are equivalent
to $L^{p}\left(\Omega;\text{d}\nu\right)$ for any $p>1.$
\end{lem}

\begin{proof}
The formula $\Sigma_{n}=J^{T}\gamma_{0}^{-1}\dn\gamma_{0}^{-1}J$
is straightforward to verify. It follows that 
\begin{equation}
\left|\dn\right|_{F}\left(\min_{\overline{\Omega}}\lambda\left(\gamma_{0}^{-1}\right)\right)^{2}\leq\left|\Sigma_{n}\right|_{F}\leq\left|\dn\right|_{F}\left(\max_{\overline{\Omega}}\lambda\left(\gamma_{0}^{-1}\right)\right)^{2}.\label{eq:equivsigmadn}
\end{equation}
Since these two quantities are equivalent, the conclusion follows.
\end{proof}
\begin{prop}
\label{prop:polarstream}Suppose Assumptions \ref{enu:well-within},
\ref{enu:pertubative}, and \ref{enu:notlow} are satisfied. Additionally
assume that $d=2$ and for some $p>2,$
\[
\limsup_{n}\left\Vert \dn\right\Vert _{L^{p}\left(\rn\right)}<\infty.
\]
Given $\Omega^{\prime}$ a smooth domain as defined in \propref{Polar-Outline},
there holds 
\[
\int_{\Omega}\left(\left(\gamma_{n}-\gamma_{0}\right)\nabla w_{n}\cdot\nabla x_{i}\right)\phi\,\text{d}x=\int_{\Omega}\left(\left(\gamma_{n}-\gamma_{0}\right)\nabla w_{n}^{i}\cdot\nabla u_{0}\right)\phi\,\text{d}x+\int_{\Omega}r_{n}\cdot\nabla\phi\,\text{d}x
\]
with 
\[
\left\Vert r_{n}\right\Vert _{L^{1}\left(\Omega\right)}\leq C\left\Vert \dn\right\Vert _{L^{1}\left(\Omega\right)}^{1+\eta}\left(\left\Vert \nabla u_{0}\right\Vert _{L^{\infty}(K)}+\left\Vert u_{0}\right\Vert _{L^{\infty}\left(\partial\Omega^{\prime}\right)}\right),
\]
where the positive constants $C$ and $\eta$ may depend only on $\tau$,
$\Omega$, $\ensuremath{K}$, $\left\Vert \gamma_{0}\right\Vert _{W^{2,d}\left(\Omega\right)}$,
$\Lambda_{0}$, $\lambda_{0}$ and $\left\Vert \dn\right\Vert _{L^{p}\left(\qn\right)}$.
\end{prop}

The proof of this result is given in appendix~\ref{app:polarstream}.

\subsection{The non finely intertwined case.}

The main result of this section is the establishes \propref{Polar-Outline}
in the final case, namely when \enuref{intertwined} holds. Example~\ref{exa:notfinelyintertwined}
is an illustration of such a configuration.
\begin{example}
\label{exa:notfinelyintertwined} Suppose that $\Omega\subset\mathbb{R}^{d}$
is the ball $B\left(0,d\right)$ of radius $d$ centred at the origin.
Assume that $\gamma_{0}=I_{d}$. Given $\epsilon>0$, for $n\geq2$,
we set 
\[
A_{n}=\bigcup_{k=1}^{n}\left(\frac{k}{n},\frac{k}{n}+\frac{1}{n^{d+1+\epsilon}}\right)\times\left(0,1\right)^{d-1},\quad B_{n}=\bigcup_{k=1}^{n}\left(\frac{k}{n}+\frac{1}{2n},\frac{k}{n}+\frac{3}{4n}\right)\times\left(0,1\right)^{d-1},
\]
and 
\[
\gamma_{n}=\left(\left(n\frac{i-1}{d-1}+\frac{d-i}{d-1}\right)\delta_{ij}\right)_{1\leq i,j\leq d}\text{ on }A_{n},\quad\gamma_{n}=\frac{\ln n}{n}I_{d}\text{ on }B_{n}.
\]
We have $A_{n}\cup B_{n}\subset\left(0,1\right)^{d}\subset\Omega.$
The insulating and conductive strips are separated by a distance $d\left(A_{n},B_{n}\right)\propto\frac{1}{n}.$
We have
\[
\left\Vert d_{n}\right\Vert _{L^{1}\left(A_{n}\right)}\propto\frac{1}{n^{d-1+\epsilon}},\quad\left\Vert d_{n}\right\Vert _{L^{1}\left(B_{n}\right)}\propto\frac{1}{\ln n},
\]
therefore $\left\Vert d_{n}\right\Vert _{L^{1}\left(\Omega\right)}\to0.$
We have $d\left(A_{n},B_{n}\right)>\left\Vert d_{n}\right\Vert _{L^{1}\left(A_{n}\right)}^{\tau}$
for $\tau\in\left(0,\frac{1}{d-1}\right).$We compute that $\left\Vert d_{n}\right\Vert _{L^{p}\left(A_{n}\right)}\propto n^{p-\left(d+\epsilon\right)}$
. In particular for $p=d>\frac{d}{2}$ there holds $\left\Vert d_{n}\right\Vert _{L^{p}\left(A_{n}\right)}\to0.$
Notice that the conductive strips are narrowed to accomodate the extra
integrability, whereas the insulating one are just chosen to so that
$\left\Vert d_{n}\right\Vert _{L^{1}\left(\Omega\right)}\to0$.
\end{example}

\begin{prop}
\label{prop:notwinning}Suppose assumptions \ref{enu:well-within},
\ref{enu:pertubative}, and \ref{enu:notlow} are satisfied. Suppose
additionally that for some $p>\frac{d}{2}$, 
\[
\limsup_{n}\left\Vert \dn\right\Vert _{L^{p}\left(\qn\right)}^{\frac{1}{2}}<\infty
\]
and that there exists a sequence of function $\left(\chi_{n}\right)_{n\in\mathbb{N}}\in\left(W^{1,\infty}\left(\Omega;[0,1]\right)\right)^{\mathbb{N}}$
such that $\chi_{n}\equiv0$ on $\rn$, $\chi_{n}=1$ on $\qn$ and
\[
\left\Vert \dn\right\Vert _{L^{1}\left(\Omega\right)}^{\tau}\left\Vert \nabla\chi_{n}\right\Vert _{L^{\infty}\left(\Omega\right)}<\infty,
\]
for some $\tau<\frac{1}{\left(d-1\right)}$. Given $\Omega^{\prime}$
a smooth domain as defined in \propref{Polar-Outline}, there holds
\[
\int_{\Omega}\left(\Big(\gamma_{n}-\gamma_{0}\Big)\nabla w_{n}\cdot\nabla x_{i}\right)\phi\text{\,d}x=\int_{\Omega}\left(\Big(\gamma_{n}-\gamma_{0}\Big)\nabla w_{n}^{i}\cdot\nabla u_{0}\right)\phi\text{\,d}x+\int_{\Omega}r_{n}\cdot\nabla\phi\text{\,d}x
\]
with 
\[
\left\Vert r_{n}\right\Vert _{L^{1}\left(\Omega\right)}\leq C\left\Vert \dn\right\Vert _{L^{1}\left(\Omega\right)}^{1+\eta}\left(\left\Vert \nabla u_{0}\right\Vert _{L^{\infty}(K)}+\left\Vert u_{0}\right\Vert _{L^{\infty}\left(\partial\Omega^{\prime}\right)}\right),
\]
where the positive constants $C$ and $\eta$ may depend only on $\tau$,
$\Omega$, $\ensuremath{K}$, $\left\Vert \gamma_{0}\right\Vert _{W^{2,d}\left(\Omega\right)}$
, $\Lambda_{0}$ and $\lambda_{0}$, $p$ and $\tau$ only.
\end{prop}

\begin{proof}
This a direct consequence of estimate \eqref{epsilon1linf} in \propref{remainderestimate}
and \lemref{Linfty2pD}.
\end{proof}
\begin{lem}
\label{lem:Linfty2pD} If for some $p>\frac{d}{2}$, 
\[
\limsup_{n}\left\Vert \dn\right\Vert _{L^{p}\left(\qn\right)}^{\frac{1}{2}}<\infty
\]
and if there exists a sequence of function $\left(\chi_{n}\right)_{n\in\mathbb{N}}\in\left(W^{1,\infty}\left(\Omega;[0,1]\right)\right)^{\mathbb{N}}$
such that $\chi_{n}\equiv0$ on $\rn$, $\chi_{n}=1$ on $\qn$ and
\[
\left\Vert \dn\right\Vert _{L^{1}\left(\Omega\right)}^{\tau}\left\Vert \nabla\chi_{n}\right\Vert _{L^{\infty}\left(\Omega\right)}<\infty,
\]
for some $\tau<\frac{1}{\left(d-1\right)}$ then there exists $\eta>0$
depending on $p$ and $\tau$ only such that 
\[
\left\Vert w_{n}\right\Vert _{L^{\infty}\left(\qn\right)}\leq C\left\Vert \dn\right\Vert _{L^{1}\left(\Omega\right)}^{\eta}\left(\left\Vert \nabla u_{0}\right\Vert _{L^{\infty}(K)}+\left\Vert u_{0}\right\Vert _{L^{\infty}\left(\partial\Omega^{\prime}\right)}\right),
\]
where $C$ depends on $K,\Omega$, $\Lambda_{0}$, $\lambda_{0}$
, $\left\Vert \gamma_{0}\right\Vert _{W^{2,d}\left(\Omega\right)}$,
$p$ and $\tau$ only.
\begin{proof}
We apply Stampacchia's truncation method \citep{stampacchia}. We
denote $u\to G_{k}(u)$ to be the truncation operator, i.e $G_{k}(u)=\begin{cases}
u & |u|\leq k\\
k & u>k\\
-k & u<-k
\end{cases}$ with $k>0$, and we write $\text{\ensuremath{m_{k}}}=\left\{ x\in\Omega\,:\,\left|u_{n}\right|>k\right\} $.
We test equation (\ref{eq:defwn0}) against $\chi_{n}^{2}v_{n}$,
with $v_{n}=w_{n}-G_{k}\left(w_{n}\right)$, and obtain 
\begin{align*}
 & \int_{\Omega}\gamma_{n}\nabla w_{n}\cdot\nabla\left(\chi_{n}^{2}v_{n}\right)\text{d}x\\
= & \int_{\Omega}\gamma_{n}\nabla\left(\chi_{n}v_{n}\right)\cdot\nabla\left(\chi_{n}v_{n}\right)\text{d}x-\int_{\Omega}\gamma_{n}\nabla\chi_{n}\cdot\nabla\chi_{n}v_{n}^{2}\text{d}x\\
= & \int_{\Omega}\chi_{n}\left(\gamma_{0}-\gamma_{n}\right)\nabla u_{0}\cdot\nabla\left(\chi_{n}v_{n}\right)\text{d}x+\int_{\Omega}\chi_{n}v_{n}\left(\gamma_{0}-\gamma_{n}\right)\nabla u_{0}\cdot\nabla\chi\text{d}x
\end{align*}
Write $\gamma_{n}^{+}=\max\left(\gamma_{n},\gamma_{0}\right)$. Since
$\chi\equiv0$ on $\rn$, and $\nabla\chi$ is supported on $\Omega\setminus\left(\qn\cup\rn\right)$
and $v_{n}$ is supported on $m_{k}$, we may simplify the above identity
to 
\[
\int_{\Omega}\gamma_{n}^{+}\nabla\left(\chi v_{n}\right)\cdot\nabla\left(\chi v_{n}\right)\text{d}x=\int_{m_{k}}\gamma_{0}\nabla\chi_{n}\cdot\nabla\chi_{n}v_{n}^{2}\text{d}x-\int_{m_{k}}\left(\gamma_{0}-\gamma_{n}^{+}\right)\nabla u_{0}\cdot\nabla\left(\chi_{n}v_{n}\right)\text{d}x
\]
Using Cauchy-Schwarz, we find 
\[
\left|\int_{m_{k}}\left(\gamma_{0}-\gamma_{n}^{+}\right)\nabla u_{0}\cdot\nabla\left(\chi_{n}v_{n}\right)\text{d}x\right|\leq\left(\int_{m_{k}}\dn\nabla u_{0}\cdot\nabla u_{0}\text{d}x\right)^{\frac{1}{2}}\left(\int_{\Omega}\gamma_{n}^{+}\nabla\chi_{n}\cdot\nabla\chi_{n}v_{n}^{2}\text{d}x\right)^{\frac{1}{2}},
\]
which shows that 
\[
\int_{\Omega}\lambda_{0}\left|\nabla\left(\chi v_{n}\right)\right|^{2}\text{d}x\leq\int_{\Omega}\gamma_{n}^{+}\nabla\left(\chi v_{n}\right)\cdot\nabla\left(\chi v_{n}\right)\text{d}x\leq2\left(\int_{m_{k}}\dn^{+}\nabla u_{0}\cdot\nabla u_{0}\text{d}x+\int_{m_{k}}\Lambda_{0}\left|\nabla\chi_{n}\right|^{2}v_{n}^{2}\text{d}x\right).
\]
For any $p>\frac{d}{2}$ we write using Hölder's inequality and the
fact that $\left|v_{n}\right|\leq\left|w_{n}\right|$
\begin{align*}
 & \int_{m_{k}}\dn^{+}\nabla u_{0}\cdot\nabla u_{0}\text{d}x+\int_{m_{k}}\gamma_{0}\nabla\chi_{n}\cdot\nabla\chi_{n}v_{n}^{2}\text{d}x\\
\leq & \left\Vert d_{n}\right\Vert _{L^{p}\left(\qn\right)}\left\Vert \nabla u_{0}\right\Vert _{L^{\infty}\left(K\right)}^{2}\left|m_{k}\right|^{1-\frac{1}{p}}+\left\Vert w_{n}\right\Vert _{L^{2p}\left(\Omega\right)}^{2}\Lambda_{0}\left\Vert \nabla\chi_{n}\right\Vert _{L^{\infty}\left(\Omega\right)}^{2}\left|m_{k}\right|^{\frac{p-1}{p}},
\end{align*}
Whereas for any $h>k$, thanks to the Sobolev embedding $H^{1}\left(\Omega\right)\hookrightarrow L^{q}\left(\Omega\right)$
for $q=\left(\frac{p}{p-1}+\frac{d}{d-2}\right)$ if $d>2$ and $q=\frac{2p}{p-1}+1$
if $d=2$, 
\[
\lambda_{\grave{a}}C\left(s,\Omega\right)\left|k-h\right|^{2}\left|m_{h}\right|^{\frac{2}{q}}<\lambda_{q}C\left(s,\Omega\right)\left\Vert \chi v_{n}\right\Vert _{L^{3+\frac{2}{s}}\left(m_{k}\right)}^{2}<\int_{\Omega}\lambda_{0}\left|\nabla\left(\chi v_{n}\right)\right|^{2}\text{d}x.
\]
This shows that $m_{k}=0$, for $k$ large enough, that is, 
\[
\left\Vert \chi_{n}w_{n}\right\Vert _{L^{\infty}\left(\Omega\right)}\leq C\left(\left\Vert d_{n}\right\Vert _{L^{p}\left(\qn\right)}^{\frac{1}{2}}\left\Vert \nabla u_{0}\right\Vert _{L^{\infty}\left(K\right)}+\left\Vert w_{n}\right\Vert _{L^{2p}\left(\Omega\right)}\left\Vert \nabla\chi_{n}\right\Vert _{L^{\infty}\left(\Omega\right)}\right),
\]
$C>0$ depends on $s$, $K,\Omega$, $\Lambda_{0}$ and $\lambda_{0}$
only. Thanks to estimate \eqref{Lsbound}, for any $\zeta\in\left[1,\frac{1}{\left(d-1\right)}\right)$
there holds
\[
\left\Vert w_{n}\right\Vert _{L^{2p}\left(\Omega\right)}\leq C\left\Vert \dn\right\Vert _{L^{1}(\Omega)}^{\frac{d\zeta}{2p}}\left(\left\Vert \nabla u_{0}\right\Vert _{L^{\infty}(K)}+\left\Vert u_{0}\right\Vert _{L^{\infty}\left(\partial\Omega^{\prime}\right)}\right),
\]
where $C$ depends on $\eta$, $\Omega^{\prime},$ $K,\Omega$, $\Lambda_{0}$
and $\lambda_{0}$ and $\left\Vert \gamma_{0}\right\Vert _{W^{2,d}\left(\Omega\right)}$.
Altogether, 
\begin{equation}
\left\Vert w_{n}\right\Vert _{L^{\infty}\left(\qn\right)}\leq C\left(\left\Vert d_{n}\right\Vert _{L^{p}\left(\qn\right)}^{\frac{1}{2}}+\left\Vert \dn\right\Vert _{L^{1}(\Omega)}^{\zeta}\left\Vert \nabla\chi_{n}\right\Vert _{L^{\infty}\left(\Omega\right)}\right)\left(\left\Vert \nabla u_{0}\right\Vert _{L^{\infty}(K)}+\left\Vert u_{0}\right\Vert _{L^{\infty}\left(\partial\Omega^{\prime}\right)}\right).\label{eq:boundLpq}
\end{equation}
Now, given $\tau<\frac{1}{d-1}$ and $p_{0}>\frac{d}{2}$ such that
\[
\limsup\left\Vert \dn\right\Vert _{L^{1}(\Omega)}^{\tau}\left\Vert \nabla\chi_{n}\right\Vert _{L^{\infty}\left(\Omega\right)}+\limsup\left\Vert d_{n}\right\Vert _{L^{p_{0}}\left(\qn\right)}<\infty,
\]
write
\[
\kappa=\sup_{n}\left\Vert \dn\right\Vert _{L^{1}(\Omega)}^{\tau}\left\Vert \nabla\chi_{n}\right\Vert _{L^{\infty}\left(\Omega\right)}+\left\Vert d_{n}\right\Vert _{L^{p_{0}}\left(\qn\right)},
\]
 and $p_{1}=\frac{1}{2}\min\left(\frac{d}{2}\frac{1}{\tau\left(d-1\right)},p_{0}\right)+\frac{d}{4}.$
By interpolation between $L^{1}\left(A_{n}\right)$ and $L^{p_{0}}\left(A_{n}\right)$
we have 
\[
\left\Vert d_{n}\right\Vert _{L^{p_{1}}\left(\qn\right)}^{\frac{1}{2}}\leq\left\Vert d_{n}\right\Vert _{L^{1}\left(\qn\right)}^{\theta_{1}}\kappa^{\frac{1}{2}-\theta_{1}},
\]
with $\theta_{1}=\frac{p_{0}-p_{1}}{2p_{1}\left(p_{0}-1\right)}>0$
and 
\[
\left\Vert \dn\right\Vert _{L^{1}(\Omega)}^{\frac{d\tau}{2p_{1}}}\left\Vert \nabla\chi_{n}\right\Vert _{L^{\infty}\left(\Omega\right)}\leq\left\Vert \dn\right\Vert ^{\theta_{2}}\kappa,
\]
with 
\[
\theta_{2}=\left(\frac{d}{2p_{1}}-1\right)\tau>0.
\]
Estimate \eqref{boundLpq} with $p=p_{1}$ and $\zeta=\tau$ concludes
the proof, with $\eta=\min\left(\theta_{1},\theta_{2}\right).$
\end{proof}
\end{lem}

\section{\label{sec:Properties}Properties of the polarisation tensor $M$}

Thanks to \lemref{boundaryindep}, we may consider alternative definitions
for the tensor $M.$ The most convenient is the periodic one, namely,
embedding $\Omega$ in a large cube $C$, we set 
\[
H_{\#}^{1}(C):=\left\{ \phi\in H_{\text{loc}}^{1}\left(\mathbb{R}^{d}\right)\,:\,\int_{C\setminus K}\phi\,\text{d}x=0\text{ and }\phi\quad C-\text{periodic}\right\} ,
\]
and $M_{ij}=D_{ij}-W_{ij}\in L^{2}\left(\Omega,d\mu\right)$ is the
scalar weak$^{*}$ limit of $\frac{1}{\|\dn\|_{L^{1}(\Omega)}}\left(\left(\nabla w_{n}^{i}+\mathbf{e}_{i}\right)\cdot\left(\gamma_{n}-\gamma_{0}\right)\mathbf{e}_{j}\right),$
where $w_{n}^{i}$ is be the unique weak solution to 
\begin{equation}
\int_{Q}\gamma_{n}\nabla w_{n}^{i}\cdot\nabla\phi\,\text{d}x=\int_{Q}\left(\gamma_{0}-\gamma_{n}\right)\mathbf{e}_{j}\cdot\nabla\phi\,\text{d}x\text{ for all }\phi\in H_{\#}^{1}(C).\label{eq:varfom-1}
\end{equation}
In \citep{capdeboscq-vogelius-03a} another version $\mathbb{M}$
of this tensor is introduced, and $M$ a natural extension to this
context.

Assuming $\gamma_{n}=\left(\left(\gamma_{1}-\gamma_{0}\right)1_{\qn\cup\rn}+\gamma_{0}\right)I_{d}$
for some regular functions $\gamma_{1}$ and $\gamma_{0}$, then the
tensor $\mathbb{M}$ introduced in \citep{capdeboscq-vogelius-03a}
is defined as the weak$^{*}$ limit in $L^{2}\left(\Omega,d\mu\right)$
of 
\[
\frac{1}{\left|\qn\cup\rn\right|}\left(\nabla w_{n}^{i}+\mathbf{e}_{i}\right)\cdot\mathbf{e}_{j}
\]
To compare both formulas, suppose $\gamma_{1}$ and $\gamma_{0}$
are constant. Then 
\[
\frac{1}{\|\dn\|_{L^{1}(\Omega)}}\left(\left(\nabla w_{n}^{i}+\mathbf{e}_{i}\right)\cdot\left(\gamma_{n}-\gamma_{0}\right)\mathbf{e}_{j}\right)=\frac{1}{\left|\qn\cup\rn\right|}\frac{1}{\sqrt{d}}\frac{\gamma_{1}}{\gamma_{1}^{2}+\gamma_{0}^{2}}\left(\gamma_{1}-\gamma_{0}\right)\left(\nabla w_{n}^{i}+\mathbf{e}_{i}\right)\cdot\mathbf{e}_{j},
\]
thus the two tensors are related by the simple fomula
\begin{equation}
M=\frac{1}{\sqrt{d}}\frac{\gamma_{1}}{\gamma_{1}^{2}+\gamma_{0}^{2}}\left(\gamma_{0}-\gamma_{1}\right)\mathbb{M},\label{eq:CVvsCO}
\end{equation}
and most properties can be read directly from \citep{capdeboscq-vogelius-06}
with the appropriate changes.
\begin{lem}[{\citep[theorem 1]{capdeboscq-vogelius-03a}}]
The entries of the polarisation tensor $M$ satisfies $M_{ij}=M_{ji}$
$\mu$-almost everywhere in $\Omega.$
\end{lem}

\begin{lem}[{See \citep[Lemma 4][]{capdeboscq-vogelius-06}}]
\label{lem:-ReussBounds} For every $\phi\in C_{c}^{1}\left(C\right),$
$\phi\geq0$, and every $\zeta\in\mathbb{R}^{d},$ there holds
\begin{align*}
\int_{\Omega}W\zeta\cdot\zeta\phi d\mu & =\frac{1}{\left\Vert \dn\right\Vert _{L^{1}\left(\Omega\right)}}\int_{\Omega}\dn^{\prime}\zeta\cdot\zeta\ensuremath{\phi}\text{\,d}x\\
 & -\frac{1}{\left\Vert \dn\right\Vert _{L^{1}\left(\Omega\right)}}\min_{u\in H_{\#}^{1}\left(C\right)^{d}}\int_{\Omega}\gamma_{n}\left(\nabla u-\gamma_{n}^{-1}\left(\gamma_{n}-\gamma_{0}\right)\zeta\right)\cdot\left(\nabla u-\gamma_{n}^{-1}\left(\gamma_{n}-\gamma_{0}\right)\zeta\right)\phi\text{d}x+o\left(1\right),
\end{align*}
with 
\[
d_{n}^{\prime}=\left(\gamma_{n}-\gamma_{0}\right)\gamma_{n}^{-1}\left(\gamma_{n}-\gamma_{0}\right)=d_{n}-2\gamma_{0}\geq0.
\]
In particular, the tensor $M$ is positive semi-definite and satisfies
\[
0\leq W\leq I_{d}\text{ \ensuremath{\mu}}\text{ a.e. in }\Omega.
\]
 If $\gamma_{n}$ and $\gamma_{0}$ are multiples of the identity
matrix, that is, the material is isotropic, then 
\[
0\leq W\leq\frac{1}{\sqrt{d}}I_{d}\text{ \ensuremath{\mu}}\text{ a.e. in }\Omega.
\]
\end{lem}

\begin{proof}
The derivation of the identity is, mutatis mutandis, done in \citep[lemma 4]{capdeboscq-vogelius-06}.
Choosing $u=0$, we find 
\begin{align*}
 & \frac{1}{\left\Vert \dn\right\Vert _{L^{1}\left(\Omega\right)}}\min_{u\in H_{\#}^{1}\left(C\right)^{d}}\int_{\Omega}\gamma_{n}\left(\nabla u-\gamma_{n}^{-1}\left(\gamma_{n}-\gamma_{0}\right)\zeta\right)\cdot\left(\nabla u-\gamma_{n}^{-1}\left(\gamma_{n}-\gamma_{0}\right)\zeta\right)\phi\text{\,d}x\\
\leq & \frac{1}{\left\Vert \dn\right\Vert _{L^{1}\left(\Omega\right)}}\min_{u\in H_{\#}^{1}\left(C\right)^{d}}\int_{\Omega}\dn^{\prime}\phi\text{\,d}x,
\end{align*}
and therefore 
\[
\int_{\Omega}W\zeta\cdot\zeta\phi d\mu\geq0.
\]
Since the second term is negative, we find 
\[
\int_{\Omega}W\zeta\cdot\zeta\phi d\mu\leq\lim_{n\to\infty}\frac{1}{\left\Vert \dn\right\Vert _{L^{1}\left(\Omega\right)}}\int_{\Omega}\phi\dn^{\prime}\zeta\cdot\zeta\text{\,d}x.
\]
 We compute 
\[
\frac{1}{\left\Vert \dn\right\Vert _{L^{1}\left(\Omega\right)}}\int_{\Omega}\phi\dn^{\prime}\zeta\cdot\zeta\text{\,d}x=\int_{\Omega}\phi\frac{\dn^{\prime}\zeta\cdot\zeta}{\left|\dn\right|_{F}}\frac{\left|\dn\right|_{F}}{\left\Vert \dn\right\Vert _{L^{1}\left(\Omega\right)}}\text{\,d}x\leq\int_{\Omega}\phi\frac{\dn^{\prime}\zeta\cdot\zeta}{\left|\dn^{\prime}\right|_{F}}\frac{\left|\dn\right|_{F}}{\left\Vert \dn\right\Vert _{L^{1}\left(\Omega\right)}}\text{\,d}x,
\]
and if $\lambda_{1}\leq\ldots\leq\lambda_{d}$ are the eigenvalues
of $d_{n}^{\prime}$ at $x$, 
\[
\frac{\dn^{\prime}\zeta\cdot\zeta}{\left|\dn^{\prime}\right|_{F}}\leq\left|\zeta\right|^{2}\frac{\lambda_{d}}{\sqrt{\sum_{i=1}^{d}\lambda_{i}^{2}}}\leq\left|\zeta\right|^{2}\begin{cases}
1 & \text{in general,}\\
\frac{1}{\sqrt{d}} & \text{if }\lambda_{1}=\ldots=\lambda_{d}
\end{cases}.
\]
All eigenvalues are equal when $\gamma_{0}$ and $\gamma_{n}$ are
isotropic, therefore 
\[
\int_{\Omega}W\zeta\cdot\zeta\phi d\mu\leq\lim_{n\to\infty}\frac{1}{\left\Vert \dn\right\Vert _{L^{1}\left(\Omega\right)}}\int_{\Omega}\phi\dn^{\prime}\zeta\cdot\zeta\text{\,d}x\leq C\left|\zeta\right|^{2}\int_{\Omega}\phi\text{d}\mu,
\]
with $C=1$ in general and $C=d^{-\frac{1}{2}}$ in isotropic media.
\end{proof}

\section{\label{sec:An-example}An example}

We revisit an example already considered in \citep{brulh-hanke-vogelius-03,capdeboscq-vogelius-04},
namely, elliptic inclusions. In a domain 
\[
\Omega=\left\{ (x,y)\subset\mathbb{R}^{2}:\frac{x^{2}}{\cosh^{2}(2)}+\frac{y^{2}}{\sinh^{2}(2)}\leq1\right\} ,
\]
consider heterogeneities in a homogeneous medium located in the set

\[
E_{n}=\left\{ (x,y)\subset\mathbb{R}^{2}:\frac{x^{2}}{\cosh^{2}\left(n^{-1}\right)}+\frac{y^{2}}{\sinh^{2}\left(n^{-1}\right)}\leq1\right\} ,
\]
which collapses to the line segment $(-1,1)\times\left\{ 0\right\} $
as $n\to\infty.$ Consider an isotropic inhomogeneity, with conductivity
\[
\gamma_{n}\left(x\right)=\begin{cases}
1 & x\in\Omega\setminus Q_{n}\\
\lambda_{n} & x\in Q_{n},
\end{cases}
\]
where $\lambda_{n}\in\left(0,1\right)\cup\left(1,\infty\right)$.
In this case, 
\[
d_{n}=\left(\lambda_{n}+\lambda_{n}^{-1}\right)I_{2}.
\]
and $\left\Vert d_{n}\right\Vert _{L^{1}\left(\Omega\right)}\to0$
means $\max\left(n^{-1}\lambda_{n},n^{-1}\lambda_{n}^{-1}\right)\to0.$
The solution $u_{n}^{i}$ to the equation 
\begin{eqnarray}
-\nabla\cdot\big(\gamma_{n}\nabla u_{n}^{i}\big) & = & 0\quad\text{in}\quad\Omega\nonumber \\
u_{n}^{i} & = & x_{i}\quad\text{on}\quad\partial\Omega\label{eq:expc-1}
\end{eqnarray}
can be computed explicitly in elliptic coordinates. In particular
we find that
\[
\frac{1}{\left|d_{n}\right|_{F}}\left(1-\gamma_{n}\right)\partial_{x_{j}}w_{n}^{i}=\frac{1}{\sqrt{2}}\frac{\lambda_{n}}{1+\lambda_{n}^{2}}\left(1-\gamma_{n}\right)1_{E_{n}}\left(\partial_{x_{j}}u_{n}^{i}-\delta_{ij}\right)=\delta_{ij}\ell_{n}^{i}1_{E_{n}},
\]
with 
\begin{align*}
\ell_{n}^{1} & =O\left(\frac{\lambda_{n}}{n}\right)\text{ and }\ell_{n}^{2}=\frac{1}{\sqrt{2}}+O\left(\frac{\lambda_{n}}{n}\right)\text{ when }\lambda_{n}>1,\\
\ell_{n}^{1} & =O\left(\frac{1}{n^{2}}\right)\text{ and }\ell_{n}^{2}=O\left(\frac{1}{n\lambda_{n}}\right)\text{ when }0<\lambda_{n}<1,
\end{align*}
As a consequence, when $n\lambda_{n}\to0$ with $\lambda_{n}\to\infty$
\[
W=\begin{pmatrix}0 & 0\\
0 & \frac{1}{\sqrt{2}}
\end{pmatrix},\quad D=\begin{pmatrix}\frac{1}{\sqrt{2}} & 0\\
0 & \frac{1}{\sqrt{2}}
\end{pmatrix}\quad M=\begin{pmatrix}\frac{1}{\sqrt{2}} & 0\\
0 & 0
\end{pmatrix},
\]
Whereas when $\lambda_{n}\to0,$we obtain 
\[
W=\begin{pmatrix}0 & 0\\
0 & 0
\end{pmatrix},\quad D=-\begin{pmatrix}\frac{1}{\sqrt{2}} & 0\\
0 & \frac{1}{\sqrt{2}}
\end{pmatrix}\quad M=-\begin{pmatrix}\frac{1}{\sqrt{2}} & 0\\
0 & \frac{1}{\sqrt{2}}
\end{pmatrix},
\]
and both results corresponds extreme cases with respect to the isotropic
pointwise bounds derived in \lemref{-ReussBounds}.
\begin{acknowledgement*}
This study contributes to the IdEx Université de Paris ANR-18-IDEX-0001.
\end{acknowledgement*}
\bibliographystyle{plainnat}
\bibliography{bibinclusions}

\appendix

\specialsection{\label{app:AppendixA}Additional proofs}
\begin{proof}[Proof of \lemref{defmu}]
The convergence \eqref{defnmu} is a direct consequence of the Banach--Alaoglu's
theorem and the continuous embedding between $L^{1}\left(\Omega\right)\hookrightarrow C^{0}\left(\overline{\Omega}\right)^{*}$,
where we have identified the continuous dual space of $C^{0}\left(\overline{\Omega}\right)$
as the space of bounded Radon measures on $\Omega$. We know from
\eqref{notlowalpha} that $\left|\left(\gamma_{0}-\gamma_{n}\right)_{ij}\right|\leq\left|d_{n}\right|_{F},$therefore
$\left\Vert \left(\gamma_{n}-\gamma_{0}\right)_{ij}\right\Vert _{L^{1}(\Omega)}\leq\left\Vert \dn\right\Vert _{L^{1}(\Omega)}$.
We may extract a subsequence in which
\[
\frac{1}{\left\Vert \dn\right\Vert _{L^{1}(\Omega)}}\left(\gamma_{n}-\gamma_{0}\right)_{ij}\overset{*}{{\rightharpoonup}}\text{d}\mathcal{D}_{ij}
\]
in the space of bounded vector Radon measures.
\begin{align*}
\int_{\Omega}\phi\,d\mathcal{D}_{ij} & =\lim_{n\to\infty}\int_{\Omega}\frac{1}{\left\Vert \dn\right\Vert _{L^{1}(\Omega)}}\left(\gamma_{0}-\gamma_{n}\right)_{ij}\phi\,\text{d}x\\
 & \leq\lim_{n\to\infty}\int_{\Omega}\frac{1}{\left\Vert \dn\right\Vert _{L^{1}(\Omega)}}\left|d_{n}\right|_{F}\phi\,\text{d}x\\
 & \leq\lim_{n\to\infty}\left(\int_{\Omega}\frac{1}{\left\Vert \dn\right\Vert _{L^{1}(\Omega)}}\left|d_{n}\right|_{F}\phi^{2}\,\text{d}x\right)^{\frac{1}{2}}\\
 & =\left(\int_{\Omega}\phi^{2}\,\text{d}\mu\right)^{\frac{1}{2}},
\end{align*}
where we used Cauchy-Schwarz in the penultimate line. It follows that
the functional 
\[
\phi\to\int_{\Omega}\phi\cdot\text{d}\mathcal{D}_{ij}
\]
 may be extended to a bounded linear functional on $\left[L^{2}(\Omega,\text{d}\mu)\right]^{d}$.
Hence, by Riesz's Representation Theorem, we may identify 
\[
\text{d}\mathcal{D}_{ij}=D_{ij}\text{d}\mu
\]
for some function $D_{ij}\in L^{2}(\Omega,\text{d}\mu)$, which is
our statement.
\end{proof}

\specialsection{\label{app:proofape}Proof of \propref{EnergyEstimateGal}}
\begin{proof}
We write 
\[
d_{n}^{\prime}=\left(\gamma_{n}-\gamma_{0}\right)\gamma_{n}^{-1}\left(\gamma_{n}-\gamma_{0}\right),
\]
and note that $d_{n}^{\prime}\leq d_{n}.$ Note that $w_{n}$ is the
unique minimiser over $X$ of the functional 
\[
J(w)=\int_{\Omega}\gamma_{n}\left(\nabla w+\gamma_{n}^{-1}\left(\gamma_{n}-\gamma_{0}\right)\nabla u_{0}\right)\cdot\left(\nabla w+\gamma_{n}^{-1}\left(\gamma_{n}-\gamma_{0}\right)\nabla u_{0}\right)\text{d}x,
\]
 Clearly, $J\left(w_{n}\right)\geq0$, thus 
\[
-\int_{\Omega}\gamma_{n}\nabla w_{n}\cdot\nabla w_{n}\,\text{d}x+2\int_{\Omega}\gamma_{n}\left(\nabla w_{n}+\gamma_{n}^{-1}\left(\gamma_{n}-\gamma_{0}\right)\nabla u_{0}\right)\cdot\nabla w_{n}\text{d}x+\int_{\Omega}\dn^{\prime}\nabla u_{0}\cdot\nabla u_{0}\text{d}x\geq0,
\]
which shows 
\begin{equation}
\int_{\Omega}\gamma_{n}\nabla w_{n}\cdot\nabla w_{n}\,\text{d}x\leq\int_{\Omega}\dn^{\prime}\nabla u_{0}\cdot\nabla u_{0}\,\text{d}x.
\end{equation}
Thus, as $u_{0}\in C^{1}(K)$
\[
\int_{\Omega}\gamma_{n}(x)\nabla w_{n}\cdot\nabla w_{n}\,\text{d}x\leq\left\Vert \nabla u_{0}\right\Vert _{L^{\infty}(K)}^{2}\int_{\Omega}\left|\dn\right|_{F}\,\text{d}x.
\]
We now turn to the second estimate. Using Cauchy--Schwarz we find
\begin{align}
 & \left\Vert \left(\gamma_{n}-\gamma_{0}\right)\nabla w_{n}\right\Vert _{L^{1}(\Omega)}\label{measureenergy}\\
 & =\int_{\Omega}\sqrt{\left|\left(\gamma_{n}-\gamma_{0}\right)\gamma_{n}^{-\frac{1}{2}}\gamma_{n}^{\frac{1}{2}}\nabla w_{n}\right|^{2}\text{d}x}\nonumber \\
 & \leq\sqrt{\int_{\Omega}\left|\left(\gamma_{n}-\gamma_{0}\right)\gamma_{n}^{-1}\left(\gamma_{n}-\gamma_{0}\right)\right|_{F}\text{d}x}\sqrt{\int_{\Omega}\gamma_{n}\nabla w_{n}\cdot\nabla w_{n}\text{d}x}\nonumber \\
 & \leq\left\Vert \dn\right\Vert _{L^{1}(\Omega)}\left\Vert \nabla u_{0}\right\Vert _{L^{\infty}(K)}.\nonumber 
\end{align}
Since $\frac{1}{\left\Vert \dn\right\Vert _{L^{1}(\Omega)}}\left(\gamma_{n}-\gamma_{0}\right)\nabla w_{n}$
is uniformly bounded in $L^{1}(\Omega)$, we may extract a subsequence
in which
\[
\frac{1}{\left\Vert \dn\right\Vert _{L^{1}(\Omega)}}\left(\gamma_{n}-\gamma_{0}\right)\nabla w_{n}\overset{*}{{\rightharpoonup}}\text{d}\text{\ensuremath{\mathscr{M}}}
\]
in the space of bounded vector Radon measures. Moreover, for any $\Psi\in C^{0}(\overline{\Omega};\mathbb{R}^{d})$,
\begin{eqnarray*}
\int_{\Omega}\Psi\cdot\text{d}\mathscr{M} & = & \lim_{n\to\infty}\int_{\Omega}\frac{1}{\left\Vert \dn\right\Vert _{L^{1}(\Omega)}}\left(\gamma_{n}-\gamma_{0}\right)\nabla w_{n}\cdot\Psi\text{d}x\\
 & \leq & \lim_{n}\left(\frac{1}{\left\Vert \dn\right\Vert _{L^{1}(\Omega)}}\int_{\Omega}\gamma_{n}\nabla w_{n}\cdot\nabla w_{n}\text{d}x\right)^{\frac{1}{2}}\left(\frac{1}{\left\Vert \dn\right\Vert _{L^{1}(\Omega)}}\int_{\Omega}\dn^{\prime}\Psi\cdot\Psi\text{d}x\right)^{\frac{1}{2}}\\
 & \leq & C\left(\int_{\Omega}\left|\Psi\right|^{2}\text{d}\mu\right)^{\frac{1}{2}}
\end{eqnarray*}
thanks to the estimate above. As a consequence of this estimate, it
follows that the functional 
\[
\Psi\to\int_{\Omega}\Psi\cdot\text{d}\text{\ensuremath{\mathscr{M}}}
\]
 may be extended to a bounded linear functional on $\left[L^{2}(\Omega,\text{d}\mu)\right]^{d}$.
Hence, by Riesz's Representation Theorem, we may identify 
\[
\text{d}\text{\ensuremath{\mathscr{M}}}=M\text{d}\mu
\]
for some function $\mathcal{M}\in\left[L^{2}(\Omega,\text{d}\mu)\right]^{d}$,
which is our statement.
\end{proof}

\specialsection{\label{app:polarstream}Proof of Proposition~\ref{prop:polarstream}}
\begin{rem*}
Note that if $\Omega^{\prime}$ is simply connected, $Fb_{n}=Fb_{0}=0$.
Remark that 
\[
\frac{1}{\left|\Gamma_{i}\right|}\int_{\Gamma_{i}}\gamma_{0}\nabla u_{0}\cdot n\text{d}\sigma=\int_{\Gamma_{i}}\gamma_{n}\nabla u_{n}\cdot n\text{d}\sigma.
\]
Let $I_{i}$ be the solution of 
\[
\divx\left(\gamma_{0}\nabla I_{i}\right)=0\text{ on }\Omega^{\prime}\text{ and }I_{i}=1\text{ on }\Gamma_{i}.
\]
By an integration by parts, 
\begin{align*}
\int_{\Gamma_{i}}\gamma_{0}\nabla u_{0}\cdot n\text{d}\sigma-\int_{\Gamma_{i}}\gamma_{n}\nabla u_{n}\cdot n\text{d}\sigma & =\int_{\Omega}\gamma_{0}\nabla u_{0}\cdot\nabla I_{i}\text{d}x-\int_{\Omega}\gamma_{n}\nabla u_{n}\cdot\nabla I_{i}\text{d}x.\\
 & =\int_{\Omega}gI_{i}\text{d}\sigma-\int_{\Omega}gI_{i}\text{d}\sigma\\
 & =0.
\end{align*}
Thus, imposing that $g\in H^{\frac{1}{2}}\left(\partial\Omega\right)$
is such that $Fb_{0}=0$, which corresponds to $N-1$ contraints in
an infinite dimensional space and therefore is not a loss of generality,
this implies that $Fb_{n}=0.$ We shall make that assumption in the
rest of this section.
\end{rem*}
\begin{proof}
By the inequality in (\ref{eq:equivsigmadn}), we have
\begin{eqnarray*}
\frac{\left\Vert \Sigma_{n}\right\Vert _{L^{1}(\Omega)}}{\left\Vert \dn\right\Vert _{L^{1}(\Omega)}} & \leq & \left(\max_{\overline{\Omega}}\lambda_{d}\left(\gamma_{0}^{-1}\right)\right)^{2},
\end{eqnarray*}
thus taking a convergent subsequence of $\frac{\left\Vert \Sigma_{n}\right\Vert _{L^{1}(\Omega)}}{\left\Vert \dn\right\Vert _{L^{1}(\Omega)}}\to a_{0}$
and a possible further extraction of the subsequence $\frac{1}{\left\Vert \Sigma_{n}\right\Vert _{L^{1}(\Omega)}}\Big(\sigma_{n}-\sigma_{0}\Big)\nabla\phi_{n}$,
\corref{lowcontrastbyhighcontrast-1} implies that, if $\Xi\in C^{0}\big(\overline{\Omega},\mathbb{R}^{2}\big)$
is an arbitrary vector field,
\begin{eqnarray*}
 &  & \lim_{n\to\infty}\int_{\Omega}\left(\frac{1}{\left\Vert \dn\right\Vert _{L^{1}(\Omega)}}\Big(\sigma_{0}-\sigma_{n}\Big)\nabla\varphi_{n}\cdot\Xi\right)\,\text{d}x\\
 & = & \lim_{n\to\infty}\int_{\Omega}\left(\frac{\left\Vert \Sigma_{n}\right\Vert _{L^{1}(\Omega)}}{\left\Vert \dn\right\Vert _{L^{1}(\Omega)}}\frac{1}{\left\Vert \dn\right\Vert _{L^{1}(\Omega)}}\Big(\sigma_{n}-\sigma_{0}\Big)\nabla\varphi_{n}\cdot\Xi\right)\,\text{d}x\\
 & = & a_{0}\lim_{n\to\infty}\int_{\Omega}\left(\frac{1}{\left\Vert \dn\right\Vert _{L^{1}(\Omega)}}\Big(\sigma_{0}-\sigma_{n}\Big)\nabla\varphi_{n}\cdot\Xi\right)\,\text{d}x\\
 & = & a_{0}\int_{\Omega}\tilde{N}\nabla\psi_{0}\cdot\Xi\,\text{d}\nu\\
 & = & \int_{\Omega}N\nabla\psi_{0}\cdot\Xi\,\text{d}\mu.
\end{eqnarray*}
Where $N=a_{0}\frac{\text{d\ensuremath{\nu}}}{\text{d}\mu}\tilde{N}$
belongs to $L^{2}(\Omega;\text{d\ensuremath{\mu)}}$. Alternatively
testing against $\left(J^{T}\gamma_{0}\right)\Xi$ we find 
\begin{align*}
 & \int_{\Omega}\frac{1}{\left\Vert \dn\right\Vert _{L^{1}(\Omega)}}\Big(\sigma_{0}-\sigma_{n}\Big)\nabla\psi_{n}\cdot\left(J^{T}\gamma_{0}\right)\Xi\,\text{d}x\\
= & \int_{\Omega}\frac{1}{\left\Vert \dn\right\Vert _{L^{1}(\Omega)}}\left(J^{T}\gamma_{0}^{-1}\Big(\gamma_{n}-\gamma_{0}\Big)\gamma_{n}^{-1}J\right)\left(J^{T}\gamma_{n}\nabla u_{n}\right)\cdot\left(J^{T}\gamma_{0}\right)\Xi\,\text{d}x\\
= & \int_{\Omega}\frac{1}{\left\Vert \dn\right\Vert _{L^{1}(\Omega)}}J^{T}\gamma_{0}^{-1}\Big(\gamma_{0}-\gamma_{n}\Big)\nabla u_{n}\cdot J^{T}\gamma_{0}\Xi\,\text{d}x.\\
= & \int_{\Omega}\frac{1}{\left\Vert \dn\right\Vert _{L^{1}(\Omega)}}\Big(\gamma_{0}-\gamma_{n}\Big)\nabla u_{n}\cdot\Xi\,\text{d}x
\end{align*}
whereas 
\begin{align*}
 & \int_{\Omega}\frac{1}{\left\Vert \dn\right\Vert _{L^{1}(\Omega)}}\Big(\sigma_{n}-\sigma_{0}\Big)\nabla\psi_{0}\cdot\left(J^{T}\gamma_{0}\right)\Xi\,\text{d}x\\
 & \int_{\Omega}\frac{1}{\left\Vert \dn\right\Vert _{L^{1}(\Omega)}}\gamma_{0}\gamma_{n}^{-1}\Big(\gamma_{0}-\gamma_{n}\Big)\nabla u_{0}\cdot\Xi\,\text{d}x.\\
 & \int_{\Omega}\frac{1}{\left\Vert \dn\right\Vert _{L^{1}(\Omega)}}\left(\Big(\gamma_{0}-\gamma_{n}\Big)+\dn\right)\nabla u_{0}\cdot\Xi\,\text{d}x.
\end{align*}
We write $\mathcal{D}$ as the limit limiting tensor corresponding
to $\left\Vert \dn\right\Vert _{L^{1}(\Omega)}^{-1}\dn$ in $L^{2}(\Omega,d\mu)^{d\times d}$,
that is, 
\[
\lim_{n\to\infty}\int_{\Omega}\frac{\dn}{\left\Vert \dn\right\Vert _{L^{1}(\Omega)}}\nabla u_{0}\cdot\Xi\,\text{d}x=\int_{\Omega}\mathcal{D}\nabla u_{0}\cdot\Xi\,\text{d}\mu
\]
Altogether, we have obtained 
\begin{align*}
\lim_{n\to\infty}\int_{\Omega}\frac{1}{\left\Vert \dn\right\Vert _{L^{1}(\Omega)}}\Big(\gamma_{0}-\gamma_{n}\Big)\nabla w_{n}\cdot\Xi\text{dx} & =-\int_{\Omega}\mathcal{D}\nabla u_{0}\cdot\Xi\,\text{d}\mu+\int_{\Omega}N\nabla\psi_{0}\cdot\left(J^{T}\gamma_{0}\right)\Xi\,\text{d}\mu\\
 & =\int_{\Omega}\left(\left(\gamma_{0}J\right)N\left(\gamma_{0}J\right)^{T}-\mathcal{D}\right)\nabla u_{0}\cdot\Xi\,\text{d}\mu
\end{align*}
which is concludes our proof.
\end{proof}

\end{document}